\documentclass[11pt,a4paper,reqno]{amsart}

\usepackage{a4wide}
\usepackage{graphicx}
\usepackage{latexsym}
\usepackage{epsfig}
\usepackage{amssymb}
\usepackage{amstext}
\usepackage{amsgen}
\usepackage{amsxtra}
\usepackage{amsgen}
\usepackage{amsthm}
\usepackage{tcolorbox}
\usepackage{mathrsfs,dsfont,mathtools}
\usepackage[T1]{fontenc}
\usepackage{amsmath}

\usepackage{color}
\usepackage[colorlinks=true, linkcolor=blue, citecolor=red]{hyperref}

\usepackage{comment}

\usepackage[pagewise]{lineno}%\linenumbers

\newcommand{\mres}{\mathbin{\vrule height 1.6ex depth 0pt width
0.13ex\vrule height 0.13ex depth 0pt width 1.3ex}}

\newtheorem{thm}{Theorem}[section]
\newtheorem*{thm*}{Theorem}
\newtheorem{prop}[thm]{Proposition}%[section]
\newtheorem{lemma}[thm]{Lemma}
\newtheorem{cor}[thm]{Corollary}

\theoremstyle{remark}
\newtheorem{rem}[thm]{Remark}

\theoremstyle{definition}

\newtheorem{definition}[thm]{Definition}

\newtheorem{assumption}[thm]{Assumption}

\numberwithin{equation}{section}

\newcommand{\R}{\mathbb{R}}

\newcommand{\leb}{\mathscr{L}}

\DeclareMathOperator{\essinf}{ess \, inf}

\renewcommand{\d}{\, \mathrm{d} }
\newcommand{\der}{\mathrm{d}}

\newcommand{\loc}{\mathrm{loc}}
\usepackage{xcolor}
\usepackage{mathtools}

\DeclarePairedDelimiter\ceil{\lceil}{\rceil}

\DeclareMathOperator{\supp}{\mathrm{supp}}

\DeclareMathOperator{\argmin}{arg\,min}

\begin{document}
\title[Segregated solutions of a cross-diffusion system]{Segregated solutions of a degenerate \\ cross-diffusion system with drifts}

\author[F.~Santambrogio]{Filippo Santambrogio}
\address[F.~Santambrogio]{ Universit\'e Lyon 1, Ecole Centrale de Lyon, INSA Lyon, Universit\'e Jean Monnet, CNRS, ICJ, UMR 5208, Villeurbanne, France}\email{santambrogio@math.univ-lyon1.fr}

\author[S.~M.~Schulz]{Simon M.~Schulz}
\address[S.~M.~Schulz]{Laboratoire de Math\'ematiques de Versailles, UMR 8100 CNRS, UVSQ, Universit\'e Paris-Saclay, 45 Av.~des \'Etats-Unis, 78000 Versailles Cedex, France}\email{simon.schulz@uvsq.fr}



%\simon{variational methods applied to pde, Initial value problems for second-order parabolic systems , nonlinear parabolic pdes, quasilinear parabolic equations with p-laplace, functions of bounded variation in one dimension}

\begin{abstract}
 We prove the global existence of segregated weak solutions of a one-dimensional degenerate cross-diffusion system with independent drifts, which is endowed with a Wasserstein gradient flow structure. We argue by a Lagrangian formulation written in terms of the (pseudo-)inverse for the cumulative mass function of the sum of the species, which we solve by a Minimising Movement Scheme in the setting of $L^2 \cap BV_\loc$. This Lagrangian problem gives rise to a parabolic PDE similar to a $p$-Laplace equation, with typical range $p \in (-\infty,1)$. We employ monotonicity methods \emph{\`a la} Minty--Browder to obtain strong convergence and pass to the limit $\tau \to 0$ in the time-step of the discrete scheme. Our contribution simultaneously treats all porous medium degeneracies, the log-entropy, and fast diffusions of index $\alpha \in (\frac{1}{3},1)$, thereby complementing the recent results \cite{FilippoCharles,GuyAlparII,GuyAlpar,Jakub}, and the prior work \cite{AlparInwon}. 
\end{abstract}

\maketitle

\setcounter{tocdepth}{1}
\tableofcontents

\section{Introduction}

We study the following cross-diffusion system in one spatial dimension $(t,x) \in (0,\infty)\times\mathbb{R}$: 
\begin{equation}\label{eq:main}
  \left\lbrace \begin{aligned}
       & \partial_t \varrho = \partial_x \big( \varrho \partial_x ( f'(\varrho+\mu) +  V) \big), \\ 
       & \partial_t \mu = \partial_x \big( \mu \partial_x (f'(\varrho+\mu) + W) \big), 
   \end{aligned} \right. 
\end{equation}
where the terms $V,W$ are given Lipschitz drifts depending only on $x$, and $f:(0,\infty) \to \mathbb{R}$ is a given function of the sum. In the sequel, we prove the existence of \emph{segregated} solutions $\varrho,\mu$ belonging to the space of probability measures \textit{i.e.}~$\varrho,\mu$ are mutually singular; more precisely, they will be absolutely continuous measures such that $\varrho(x)\mu(x)=0$ $\leb$-a.e.~$x$, \textit{cf.}~\eqref{eq:how rho and mu depend on b ish}. Our strategy relies on a Lagrangian formulation of the problem via the  pseudo-inverse of the cumulative mass function for the sum $S=\varrho+\mu$, which itself satisfies 
\begin{equation}\label{eq:sum eq}
  \partial_t S = \partial_x \big( S \partial_x ( f'(S) ) \big)  + \partial_x \big( \varrho \partial_x V + \mu \partial_x W \big). 
\end{equation}

\subsection{Context and novelty}

In this portion of the paper, we give a brief overview of cross-diffusion systems and outline the ``state-of-the-art'' for system \eqref{eq:main}. 

\subsubsection{Background on cross-diffusion systems} 

Systems of advection-diffusion equations arise naturally when studying collective behaviour in the
biological and social sciences. Instances of such systems can be found in the modelling of multiple chemotactic populations in competition for nutrient \cite{Conca,Espejo}, tumour growth \cite{JungelStelzer,Bubba}, population biology \cite{DesvillettesLepoutre,GurtinPipkin}, 
neural networks \cite{JungelPortisch}, and semiconductors \cite{ChenJungeliii,Markowich}.

One of the most pertinent features of cross-diffusion systems is their relevance in describing cell sorting. This is a reorganisation process in which cells of different categories regroup into subregions with clearly defined boundaries \cite{Kang,SKT}; thereby giving rise to segregation effects. This biological phenomenon is related to the inhibition/activation of growth whenever two populations occupy the same environment \cite{BertschHilII}. These considerations make the study of segregated solutions of cross-diffusion systems a central theme, both for theory and applications.

A complete well-posedness theory for cross-diffusion systems involving transport terms is currently out of reach. This is partly due to the lack of maximum principles, which means that the classical parabolic theory \cite{Amann,russia} no longer applies. Furthermore, the aforementioned segregation effects  between species naturally lead to the formation of boundaries and sharp interfaces, which make the analysis of such equations delicate. This is exacerbated by the presence of degenerate
diffusions, which cause a loss of strict parabolicity and can give rise to singularities \cite{MichelPierre}. In turn, one often tailors the analysis to low-regularity settings which can accommodate for the presence of jumps and discontinuities. And indeed, when the drift terms in \eqref{eq:main} are replaced with reaction terms, an existence
theory was obtained in a one-dimensional $BV$ setting in \cite{SplittingScheme}; this setting was also highlighted in \cite{BurgerCarrillo,BurgerMarco}, while for the Sobolev setting we refer to \textit{e.g.}~\cite{LucaCross,BertschHilI,ChenJungeli,ChenJungelii,Jacobs}. However, it is fair to say that, while the functional framework of $BV$ seems suitably general for analysis, it does not exploit any special structure of the cross-diffusion system at hand.

Some cross-diffusion systems are intrinsically endowed with a gradient flow structure, meaning that they can be recast as the dynamic time-evolution of a minimisation problem associated to a particular energy functional. For instance, the system \eqref{eq:main} can be rewritten as
\begin{equation}\label{eq:reformulation gf}
     \partial_t         \varrho = \partial_x \left( \varrho \partial_x \frac{\delta \mathrm{E}}{\delta\varrho}  \right),   \quad \partial_t         \mu = \partial_x \left( \mu\partial_x \frac{\delta \mathrm{E}}{\delta\mu}  \right), 
 \end{equation}
where $\frac{\delta \mathrm{E}}{\delta\varrho}$ and $\frac{\delta \mathrm{E}}{\delta\mu}$  denote the first variation of $\mathrm{E}$ with respect to its two variables $\varrho$ and $\mu$, respectively, where the energy $\mathrm{E}$ is given by
\begin{equation}\label{eq:E functional}
    \mathrm{E}[\varrho,\mu] := \int_{\mathbb{R}} f(\varrho+\mu) \d x + \int_{\mathbb{R}} \big( \varrho V + \mu W \big) \d x.
\end{equation}
While there exists ample literature concerning the well-posedness of Wasserstein gradient flows \cite{AGS,Filippo} in the case of single population densities, these results do not generally apply to the case of systems such as \eqref{eq:reformulation gf}, due to the lack of $\lambda$-geodesic convexity (unless the system is actually diagonal, thanks to a clever construction in \cite{BeckZizza} which shows necessary conditions for $\lambda$-geodesic convexity in $W_2\times W_2$). Nevertheless, this gradient flow approach to cross-diffusion systems was successfully employed in \cite{MarcoSimoneI,Trio} for systems of non-local interaction equations with
two species and non-symmetric cross-interactions, and was first used in a system with cross-diffusion terms
but no transport terms in \cite{LaurencotMuskat}. Other results focus on diagonal diffusion and 
bounded domains \cite{CarlierLaborde1,CarlierLaborde2}, small cross-diffusion \cite{LucaMariaYves,LucaCross}, convergence to equilibrium \cite{BeckZizza}, or triangular structures \cite{DesvillettesLaurencot,DesvillettesTrescases}. In some cases, one may interpret the system of equations at hand as a perturbation of a gradient flow associated with a convex functional \cite{MariaAsymptotic,DucasseYoldasSantambrogio}, which sometimes makes them amenable to the boundedness-by-entropy method \cite{BoundEntropy,JungelBoundedness}; see also \cite{Martin} for a non-local version involving infinitely many species. All things considered, despite the simple gradient flow structure of \eqref{eq:main}, there is still much to discover concerning the behaviour of this system of equations.

\subsubsection{State-of-the-art for system \eqref{eq:main}}\label{sec:state of the art}
The simplest case of \eqref{eq:main} arises when $\partial_x V = \partial_x W = 0$, which makes the system of equations purely diffusive. We can then compute, for
\begin{equation}\label{eq:E functional b zero}
    \mathrm{F}[S] = \int_{\mathbb{R}} f(S) \d x, 
\end{equation}
the first variation as $\frac{\delta  \mathrm{F}}{\delta S} [S] = f'(S)$. In this instance of the problem, the evolution of the sum $S$ can be recast as the closed-form gradient flow: 
\begin{equation}\label{eq:eqn for S gf}
    \partial_t S = \partial_x(S \partial_x \frac{\delta  \mathrm{F}}{\delta S}[S]), 
\end{equation}
and existence for \eqref{eq:eqn for S gf} follows directly from the well-established Wasserstein gradient flow theory \cite{AGS}. Once existence of $S$ is known, one can go back to the individual equations for $\varrho,\mu$ and treat them as independent linear scalar conservation laws: $\partial_t \varrho = \partial_x (\varrho \partial_x f'(S))$ and similarly for $\mu$. Since the density $S$ that is obtained as a solution to \eqref{eq:eqn for S gf} is sufficiently smooth, this linear equation is well-posed, which allows to obtain existence of $\varrho$ and $\mu$ and also to verify that the $\varrho,\mu$ obtained in this manner satisfy $\varrho + \mu = S$. This purely diffusive problem and its segregating effects had already been studied in \cite{BertschBV}; \textit{cf.}~system (1.6) therein. When the drifts are not null, and not identical, the situation is  more intricate.

A first general existence result for system \eqref{eq:main} was given in \cite{AlparInwon}, which involved very restrictive assumptions on the ordering of the potentials $V,W$ and on the initial configuration so as to obtain segregated solutions for all times. The problem then stayed with no progress for some years, until the next breakthrough, obtained in \cite{GuyAlpar}, where global existence to \eqref{eq:main} with periodic boundary conditions, general Lipschitz potentials $V,W$, and a broad class of initial data was shown for the specific case $f(s) = s\log s - s$. The recent contribution \cite{FilippoCharles} then improved the result of \cite{GuyAlpar}, and proved existence of periodic solutions for a relatively broad class of (``totally mixed'') initial data with $f(s) \sim s^\alpha$ for $0<\alpha<1$; the case $\alpha=1$ coincides precisely with the choice of $f$ in \cite{GuyAlpar}. Both of these recent contributions employ a formulation of the problem where the main quantities that are studied are the sum $\varrho+\mu$ and the quotient $\varrho/\mu$. Therein, the authors obtain 
%$L^p$-boundedness of weighted derivative terms of the form $\sqrt{\varrho}\partial_x(f'(S)), \sqrt{\mu}\partial_x(f'(S))$ as well as 
$BV$ estimates for the logarithm of the quotient $\varrho/\mu$ and deduce suitable bounds on $\varrho$ and $\mu$ from this and $H^1$ bounds on functions of $S$. We highlight that in one dimension, $BV$ functions are $L^\infty$, and that the boundedness from above and below of this logarithm corresponds to the case of \emph{totally mixed solutions}, with overlapping supports. This structure condition was then relaxed in \cite{GuyAlparII} (relying on the ratio of each density with respect to the sum, instead of using the logarithm). A much more general existence result was subsequently obtained in \cite{Jakub} using a compensated compactness approach; this contribution can accommodate for all choices of fast diffusion, porous medium degeneracy, and the log-entropy.

\subsubsection{Novelty}
The focus of the present manuscript is not to obtain as general an existence result as possible, but rather to study the segregation dynamics encoded in \eqref{eq:main} using variational techniques. In particular, we are concerned with the \emph{preservation of segregation} from segregated initial data (which was not studied in the aforementioned works \cite{FilippoCharles,GuyAlparII,GuyAlpar,Jakub}). This is a natural case to consider from a thermodynamic/energy-minimising perspective, since mixing is often energetically less favourable than staying segregated. Our contribution is closer in spirit to that of \cite{AlparInwon}, which is the only other work to specifically study segregated solutions, but the result contained in the present paper is much stronger than that of \cite{AlparInwon} as we allow for arbitrary segregated initial data and general drifts. That is to say, we do not require all the mass of $\varrho$ to be supported to the left of $\mu$ for example, and we allow part of the support of $\varrho$ to be squeezed between two components of the support of $\mu$. In terms of the the drifts, we do not require anymore their orientation to be such that it favours segregation. As a byproduct of our approach, which entails a Lagrangian formulation (see \S \ref{subsec:strategy}), we end up solving a parabolic $p$-Laplace problem with typical range of exponent $p \in (-\infty,1)$ by considering a gradient flow in the setting $L^2 \cap BV_\loc$, which is novel in and of itself.

\subsection{Set-up}

We describe our approach and the novel results contained herein in more detail. Firstly, we outline which segregation property will be preserved in the solution we build.

\subsubsection{Segregation}\label{sec:notion of segregation} 
This work is concerned with \eqref{eq:main} where one assumes a particular segregated structure on the solution. To explain this notion, we introduce the cumulative distribution function for the sum $S_t = S(t,\cdot)=\varrho(t,\cdot)+\mu(t,\cdot)$, denoted by $F_t$:
\begin{equation}\label{eq:cdf}
F_t(x) := \int_{-\infty}^x S_t(x') \d x'. 
\end{equation}
We study solutions which are such that, for $A \subseteq [0,2]$ a given measurable subset satisfying $\leb(A)=1=\leb(A^c)$, it holds 
\begin{equation}\label{eq:how rho and mu depend on b ish}
    \varrho_t(x) = S_t(x) \mathds{1}_A(F_t(x)), \qquad \mu_t(x) = S_t(x) \mathds{1}_{A^c}(F_t(x)); 
\end{equation}
where $\varrho_t=\varrho(t,\cdot),\mu_t=\mu(t,\cdot) \in \mathrm{Prob}(\mathbb{R})$. It follows that $F_t$ takes values in $[0,2]$, and  furthermore, by writing the formal change of variables $y=F_t(x)$, \textit{i.e.~}$\d y = S_t(x) \d x$, it holds 
\begin{equation*}
   1 = \int_{\mathbb{R}} \varrho_t(x) \d x = \int_{\mathbb{R}} \mathds{1}_A(F_t(x)) S_t(x) \d x = \int_{[0,2]} \mathds{1}_A(y) \d y = \leb(A), 
\end{equation*}
and similarly for $A^c$; thus the requirement $\leb(A)=1=\leb(A^c)$ is necessary. 

The condition \eqref{eq:how rho and mu depend on b ish} imposes that the solutions are \emph{segregated} in the sense: 
\begin{equation*}
  \{x \in \mathbb{R}: \, \varrho_t(x) > 0\} \cap \{x \in \mathbb{R}: \, \mu_t(x) > 0\} = \varnothing \qquad \forall t \geq 0.   
\end{equation*}
Morally speaking, the structure condition \eqref{eq:how rho and mu depend on b ish} means that the supports of $\varrho_t,\mu_t$ have a fixed \emph{ordering}, which does not change in time; however, the supports themselves do evolve in time. Note that the case studied in \cite{AlparInwon} corresponds to the case $A=[0,1]$ (with additional assumptions on $V$ and $W$).

\subsubsection{Strategy}\label{subsec:strategy} With the segregated formulation of \S \ref{sec:notion of segregation}, the final term of \eqref{eq:sum eq} equals 
\begin{equation}\label{eq:how to get b from rho and mu}
  \partial_x \big(   \varrho_t(x) \partial_x V(x) + \mu_t(x) \partial_x W(x) \big) = \partial_x \big( S_t(x) b(F_t(x),x) \big), 
\end{equation}
where we used \eqref{eq:how rho and mu depend on b ish}, and where 
\begin{equation}\label{eq:b def}
    b(y,s) := \mathds{1}_{A}(y) \partial_x V(s) + \mathds{1}_{A^c}(y) \partial_x W(s); 
\end{equation}
note that, for $V,W$ Lipschitz continuous, it holds $b \in L^\infty([0,2]\times\mathbb{R})$. It then follows that the entire system \eqref{eq:main} effectively collapses into the single equation 
\begin{equation}\label{eq:sum eq collapse}
  \partial_t (\partial_x F_t) = \partial_x \big( \partial_x F_t \cdot \partial_x ( f'(\partial_x F_t) ) \big)  + \partial_x \big( \partial_x F_t \cdot b(F_t,x) \big), 
\end{equation}
and we have that $\varrho_t,\mu_t$ given by \eqref{eq:how rho and mu depend on b ish} are solutions of the original system \eqref{eq:main}.

The equation \eqref{eq:sum eq collapse} is not in a particularly nice form, which motivates us to consider the equation satisfied by its (pseudo)-inverse function $u_t = F_t^{-1}$. As shown in \S \ref{sec:formal lagrangian}, it turns out that the equation for this inverse---which we call the \emph{Lagrangian reformulation}---has a familiar structure; in formal non-divergence form, it may be written 
\begin{equation}\label{eq:main eq pseudo}
    \partial_t u_t = f''(\frac{1}{\partial_y u_t}) \frac{\partial_{yy} u_t}{(\partial_y u_t)^3} - b(y,u_t), \qquad \qquad y \in [0,2], 
\end{equation}
with $b$ given by \eqref{eq:b def}. For $f(s) \sim s^\alpha$ with $\alpha> 0$, (with a negative sign if $\alpha<1$ so as to have a convex function $f$) one can see that the leading-order terms in \eqref{eq:main eq pseudo} formally correspond to a \emph{parabolic $p$-Laplace equation}, but with $p=1-\alpha \in (-\infty,1)$. This formulation enables us to harness powerful monotonicity methods \emph{\`a la} Minty--Browder. Furthermore, one can show that \eqref{eq:main eq pseudo} is endowed with an $L^2$-gradient flow structure (\textit{cf.}~\S \ref{sec:formal gf}). Our approach is therefore to solve \eqref{eq:main eq pseudo} by means of Minimising Movement Scheme (\textit{cf.}~\S \ref{sec:minimising movement}--\S \ref{sec:solution of lagrangian}), and to then deduce existence to \eqref{eq:sum eq collapse}---and hence \eqref{eq:main}---by translating back into the original coordinates.

\subsection{Plan of the paper and notations}

\subsubsection{Organisation of the paper} In \S \ref{sec:func} we introduce the necessary functional framework for our analysis and to state our main results. \S \ref{sec:main results} contains the statements of our main theorems. \S \ref{sec:lagrangian} explains the Lagrangian reformulation of the problem in \eqref{eq:main eq pseudo} and its $L^2$-gradient flow structure. In \S \ref{sec:minimising movement}, we write the discrete-time Minimising Movement Scheme associated to the gradient flow formulation of \eqref{eq:main eq pseudo} with fixed time-step $\tau>0$. \S \ref{sec:solution of lagrangian} is concerned with passing to the limit $\tau \to 0$ in the time-step of the discrete-time scheme and proving existence to the Lagrangian reformulation. In \S \ref{sec:segregated cross diff}, we translate the result for the Lagrangian reformulation into the original coordinates and prove existence to \eqref{eq:main}.

\subsubsection{Notations} Throughout the paper, we denote the Lebesgue measure by $\leb$ and, to avoid heavy notation, we also denote the two-dimensional Lebesgue on $(0,\infty)\times\mathbb{R}$ and the restriction of Lebesgue measure to a particular set by $\leb$, where no confusion arises. For a set $A$, we denote its complement by $A^c$. The family of probability measures on $\mathbb{R}$ is denoted $\mathrm{Prob}(\mathbb{R})$, while $\mathrm{Prob}_2(\mathbb{R})$ denotes the set of probabilities with finite second moment, \textit{i.e.}~$\varrho\in \mathrm{Prob}(\mathbb{R})$ belongs to $\mathrm{Prob}_2(\mathbb{R})$ if $\int|x|^2\d\varrho(x)<\infty$. For a given open set $U\subseteq \mathbb{R}$, we denote by $\mathcal{M}_\loc(U)$ (and $L^p_\loc(U)$, resp.) the family of locally finite measures (and locally finite $L^p$ functions, resp.) on $U$, and $BV_\loc(U)$ for functions belonging to $L^1_\loc(U)$ with derivative belonging to $\mathcal{M}_\loc(U)$. Given $v \in BV_\loc((0,2))$ we denote by $\partial_y v$ its distributional derivative (w.r.t.~the variable $y$, which is the standard choice of notation that we use in the paper for elements of $(0,2)$), while the notation $v'$ denotes the absolutely continuous part of $\partial_y v$. $W_2$ denotes the 2-Wasserstein distance on $\mathbb{R}$. Throughout, $\tau$ denotes the time-step for the variational scheme in \S \ref{sec:minimising movement}.

\section{Functional Framework}\label{sec:func}

In this section, we state some functional analytic results on the space $\mathcal{M}_\loc((0,2))$ of locally finite measures; in the sequel, we study variational problems over the space $BV_\loc((0,2))$.

\begin{definition}[The space $C_c((0,2))$ and its dual]
    We recall the space of continuous functions of compact support $
        C_c((0,2)) = \{ \varphi \in C([0,2]) : \, \supp \varphi \subset (0,2)  \}$, 
    equipped with the norm $\Vert \cdot \Vert_{L^\infty([0,2])}$. The dual $(C_c((0,2)))' = \mathcal{M}((0,2))$ is the space of signed measures on $(0,2)$, equipped with the total variation norm $
    \Vert \nu \Vert_{\mathcal{M}((0,2))} = |\nu|((0,2))$, where $|\nu| = \nu^+ + \nu^-$ and $\nu^\pm$ are nonnegative measures in $\mathcal{M}((0,2))^2$, singular to each other (i.e. concentrated on two disjoint sets), such that $\nu = \nu^+ - \nu^-$. 
\end{definition}

In the sequel, we will often work with measures which belong to $\mathcal{M}_\loc((0,2))$, \textit{i.e.~}locally finite measures on $(0,2)$: when we write $\nu \in \mathcal{M}_\loc((0,2))$ we mean $\nu \in \mathcal{M}(K) = (C(K))'$ for every compact subset $K \subset (0,2)$. Often these measures will be non-negative: we write \emph{$\nu \geq 0$} if, for all compact subsets $K\subset (0,2)$, it holds $\nu(K) \geq 0$. 
Equivalently, $\nu \geq 0$ if, for all non-negative $\varphi \in C_c((0,2))$, it holds $\int_0^2 \varphi(y) \d \nu(y) \geq 0$; note that the integral is well-defined for $\varphi \in C_c((0,2))$ and $\nu \in \mathcal{M}_\loc((0,2))$. 

We recall the following weak-* compactness result for locally finite measures; its proof is standard, and follows by compact exhaustion of $(0,2)$ and a diagonal argument.

\begin{thm}[Alaoglu's Theorem for $\mathcal{M}_\loc((0,2))$]\label{thm:alaoglu Mloc}
    Let $\{\nu_n\}_n$ be a sequence belonging to $\mathcal{M}_\loc((0,2))$ for which it holds 
   $\sup_{n\in\mathbb{N}} \Vert \nu_n \Vert_{\mathcal{M}([y_0,y_1])} \leq C_{y_0,y_1}$ for all $[y_0,y_1] \subset (0,2)$. Then, there exists a subsequence $\{\nu_{\sigma(n)}\}_n$ and $\nu \in \mathcal{M}_\loc((0,2))$ such that $\int_0^2 \varphi \d \nu_{\sigma(n)}  \to \int_0^2 \varphi \d \nu$ for all $\varphi \in C_c((0,2))$; we write \emph{$\nu_{\sigma(n)} \overset{*}{\rightharpoonup} \nu$ weakly-* in $\mathcal{M}_\loc((0,2))$}. 
\end{thm}

Exactly as it happens for finite signed measures, each $\nu \in \mathcal{M}_\loc ((0,2))$ can be uniquely decomposed (\emph{Jordan decomposition}) as $\nu=\nu^+ - \nu^-$ where $\nu^+,\nu^-\geq 0$ are singular to each other. Moreover, every $\nu \in \mathcal{M}_\loc((0,2))$ is also decomposed uniquely into its \emph{absolutely continuous and singular parts}, denoted $(\nu_a,\nu_s) \in L^1_\loc((0,2)) \times \mathcal{M}_\loc((0,2))$, 
\begin{equation}\label{eq:absolutely cts and singular decomp}
    \nu^+ = \nu_a \!\cdot\! \leb  + \nu_s. 
\end{equation}

Note that, in the above, when $\nu\geq 0$ then $\nu_a \geq 0$ $\leb$-a.e.~and $\nu_s \geq 0$.

\begin{definition}[The space $BV_\loc((0,2))$]\label{def:BVloc}
    We define the space $BV_\loc((0,2))$ to be the subset of $L^1_\loc((0,2))$ composed of those functions whose distributional derivative is a locally finite measure on $(0,2)$, \textit{i.e.}, 
    \begin{equation*}
        BV_\loc((0,2)) := \Big\{ v \in L^1_\loc((0,2)): \, \partial_y v \in \mathcal{M}_\loc((0,2))  \Big\}. 
    \end{equation*}
    For $v \in BV_\loc((0,2))$ we write $v' = (\partial_y v)_a$ in the decomposition \eqref{eq:absolutely cts and singular decomp}, \textit{i.e.}, 
    \begin{equation}\label{eq:shorthand prime}
    \partial_y v = v' \! \cdot \! \leb + (\partial_y v)_s.
    \end{equation}
\end{definition}

As a consequence of the compactness results for locally bounded measures, we note that a compactness result for sequences of functions in $BV_{\loc}$ is also available:
if $\{v_n\}_n$ is a sequence belonging to $BV_\loc((0,2))$ for which it holds $\sup_{n\in\mathbb{N}} \Vert v_n \Vert_{TV([y_0,y_1])} \leq C_{y_0,y_1}$ for all $[y_0,y_1] \subset (0,2)$, there exists a subsequence $\{v_{\sigma(n)}\}_n$ and $v \in BV_\loc((0,2))$ such that  $\partial_y v_{\sigma(n)} \overset{*}{\rightharpoonup} \partial_y v$ weakly-* in $\mathcal{M}_{\loc}((0,2))$,     and $\lim_{n\to\infty} \Vert v_{\sigma(n)} - v \Vert_{L^1([y_0,y_1])} = 0 $ for all $[y_0,y_1] \subset (0,2)$; we write \emph{$v_{\sigma(n)} \to v $ strongly in $L^1_\loc((0,2))$}.

\begin{rem}[Precise representative in $BV_\loc$]\label{rem:precise rep BV}
    Recall from \cite[Theorem 3.28 and p.139]{AFP} that, for all $u \in BV_\loc((0,2))$, there exists a unique right-continuous function $u^r$ defined everywhere in $(0,2)$ which satisfies $u^r = u$ $\leb$-a.e.; we call $u^r$ the (right-continuous) \emph{precise representative} for $u$. For this representative, it holds 
    \begin{equation*}
        u^r(y_1) - u^r(y_0) = \partial_y u([y_0,y_1)) \qquad \forall [y_0,y_1] \subset (0,2). 
    \end{equation*}
    Throughout the paper, we always identify $BV_\loc$ functions with a precise representative which, by arbitrary convention, can be chosen to be the right-continuous one (but all results would also work if choosing the left-continuous one, or the average of the right-continuous and of the left-continuous, or any other reasonable precise representative). Moreover, we use the shorthand notation $\int_{y_0}^{y_1} \partial_y u$ to denote $\partial_y u([y_0,y_1))$, whence the following version of the Fundamental Theorem of Calculus in $BV_\loc$ is satisfied: 
    \begin{equation}\label{eq:BV FTC}
        u(y_1) - u(y_0) = \int_{y_0}^{y_1} \partial_y u \qquad \forall [y_0,y_1] \subset (0,2). 
    \end{equation}
    This clarification removes all ambiguity concerning the evaluation of the integral on the right-hand side when either of the endpoints are atoms for the measure $\partial_y u$. 
\end{rem}

Finally, we introduce the function space in which we will conduct the bulk of our analysis. 

\begin{definition}[The space $X$]\label{def:space X}
    We define the function space $X := L^2([0,2]) \cap BV_\loc((0,2))$. For a sequence $\{v_n \}_n \subset X$, we say it is \emph{bounded in $X$} if $\sup_n \Vert v_n \Vert_{L^2([0,2])} < \infty$ and if, for all $[y_0,y_1] \subset (0,2)$, it holds $\sup_{n\in\mathbb{N}} \Vert v_n \Vert_{TV([y_0,y_1])} \leq C_{y_0,y_1} <\infty$. For a sequence $\{v_n\}_n\subset X$, we say that $v_n$ converges to $v$ in $X$ if $||v_n-v||_{L^1([0,2])}\to 0$ and $v_n$ is bounded in $X$. \end{definition}

We observe that bounded sets in $X$ are precompact for this convergence. Indeed, the $BV_\loc$ bound provides strong $L^1$ compactness on each interval $[y_0,y_1]$ by Helly's Theorem \cite[Theorem 3.23]{AFP} and hence, up to a subsequence, $\leb$-a.e.~convergence, and the $L^2$ bound on the whole interval $[0,2]$ provides equi-integrability of $v_n$, which transforms (via an easy application of Egoroff's theorem) the $\leb$-a.e.~convergence into strong $L^1$ convergence. We also observe that when a sequence  $v_n$ converges to $v$ in $X$, then we also have $v_n\to v$ in $L^2_{\loc}((0,2))$ and $v_n\to v$ in $L^p([0,2])$ for every $p<2$.

\section{Main Results}\label{sec:main results}

We give our notion of solution, list the assumptions on $f$ in \eqref{eq:main}, and state our main results. 

\begin{definition}[Segregated weak solution]\label{def:segregated}
    We say that a pair of curves $(\varrho_t,\mu_t)$ valued in $L^1(\mathbb{R})\cap\mathrm{Prob}(\mathbb{R})$ is a \emph{segregated weak solution} of \eqref{eq:main} if $\varrho_t,\mu_t \in L^\infty(\mathbb{R})$ for $\leb$-a.e.~$t$, there exists a measurable subset $A \subseteq [0,2]$ such that \eqref{eq:how rho and mu depend on b ish} holds, $\sqrt{\varrho_t+\mu_t}\partial_x(f'(\varrho_t+\mu_t)) \in L^2_\loc(0,\infty;L^2(\mathbb{R}))$, and for all $\varphi \in C^1_c(\mathbb{R})$, it holds for $\leb$-a.e.~$t>0$, 
    \begin{equation}\label{eq:weak form def}
        \begin{aligned}
            &\frac{\der}{\der t} \int_\mathbb{R}  \varrho_t \varphi \d x = - \int_\mathbb{R} \varrho_t \partial_x\big(   f'(\varrho_t + \mu_t)  +  V \big) \partial_x \varphi  \d x  , \\ 
            & \frac{\der}{\der t} \int_\mathbb{R}  \mu_t \varphi \d x = - \int_\mathbb{R} \mu_t \partial_x \big(   f'(\varrho_t+\mu_t)  +   W \big) \partial_x \varphi  \d x. 
        \end{aligned}
    \end{equation}
\end{definition}

    In what follows, we will assume that $f \in C([0,\infty)) \cap C^2((0,\infty))$ is strictly convex on $(0,\infty)$ and  $f(0)=0$. Associated with $f$, we define the function $\tilde f$ through
\begin{equation}\label{eq:f tilde def}
\tilde{f}(s):= sf(1/s) \qquad \text{for all } s > 0. 
\end{equation} 

\begin{rem}[Properties of $\tilde{f}$]\label{properties of tilde f}
We compute the derivatives of $\tilde f$ and obtain the explicit formulas 
\begin{equation}\label{eq:f tilde deriv expressions}
    \tilde{f}'(s) = f(1/s) - \frac{1}{s}f'(1/s), \qquad \tilde{f}''(s) = \frac{1}{s^3}f''(1/s).
\end{equation}
Since (for a $C^2$ function) strict convexity is equivalent to the second derivative not vanishing on any interval, it follows from the previous expression for $\tilde{f}''$ that $\tilde{f}$ is also strictly convex. From $f(0)=0$ and the strict convexity of $f$ we deduce $f'(s)s > f(s)-f(0)=f(s)$, whence 
\begin{equation}\label{eq:f' negative}
    \tilde f' < 0,     
\end{equation}
and hence $\tilde f$ is a strictly decreasing function. Moreover, we observe that we have 
\begin{equation}\label{ftilde'infty}
\lim_{s\to\infty}\tilde f'(s)=\lim_{s\to 0} \big( f(s)-sf'(s)\big) =0.
\end{equation}
In computing this limit, we use $f(0)=0$ and $\lim_{s\to 0}sf'(s)\le 0$. The latter is justified by the fact that, if $f'$ is unbounded close to $0^+$, then it can only be negative (as a consequence of convexity). This implies $\liminf_{s\to\infty}\tilde f'(s)\ge 0 $ which, together with $\tilde f'(s) < 0$, proves \eqref{ftilde'infty}.

The function $\tilde f$ could be bounded or not in a neighbourhood of $s=0$ but, in any case, it admits a (possibly infinite) limit for $s\to 0^+$ due to its monotonicity. We then extend $\tilde f$ to $[0,+\infty)$ by continuity, accepting the possible value $\tilde f(0)=+\infty$. 
\end{rem}

Our assumptions on $f$ are listed below. 

\begin{assumption}\label{ass:f conditions}
We assume $f \in C([0,\infty)) \cap C^2((0,\infty))$ is strictly convex on $(0,\infty)$ and  $f(0)=0$, and we assume that $f$ and $\tilde f$ satisfy
the following two conditions:
\begin{equation}\label{eq:f tilde prime blows up at zero assumption-1}\lim_{s\to 0^+}\tilde f'(s)=-\infty,
\end{equation}
and, for some positive constant $C$ and some $\alpha \in ( \frac{1}{3} , 1) $, the lower bound 
\begin{equation}\label{eq:the one third condition}
    f(s) \geq -Cs^\alpha \qquad \forall s \geq 0. 
\end{equation}
\end{assumption}

We remark that we \emph{do not impose} a quantitative convexity assumption such as $f'' > 0$ a.e.~(even for $f \in C^2$ strictly convex, its second derivative $f''$ is allowed to vanish on a set of positive Lebesgue measure, just not on an interval). Furthermore, note that the condition \eqref{eq:f tilde prime blows up at zero assumption-1} is weaker than superlinearity, \textit{i.e.}
\begin{equation}\label{eq:superlinear}
  \lim_{s \to \infty}\frac{f(s)}{s} = +\infty,
    \end{equation}
which follows from the fact that the superlinearity of $f$ is equivalent to 
$\lim_{s\to 0^+}\tilde f(s)=+\infty$.

    \begin{rem}[Admissible $f$]\label{rem:admissible f}
       Assumption \ref{ass:f conditions} allows for any $f$ of porous medium type $f(s)=s^\alpha$ for $\alpha>1$, the log-entropy $f(s)=s\log s - s$, and a large class of fast diffusions $f(s) = -s^\alpha$ for $\frac{1}{3} < \alpha < 1$. The reason for imposing $\alpha>\frac 13$ is connected with the choice of the $L^2$ space that we will use later and is standard whenever considering gradient flows in $W_2$: equations of fast diffusion type in $W_2(\mathbb{R}^d)$ can be considered as far as the exponent of the diffusion $\alpha$ satisfies $\alpha>\frac{d}{d+2}$, which is required for finiteness of the second moment of the Barenblatt solutions. 
    \end{rem}

Our first main theorem is concerned with the global existence of weak segregated solution of the cross-diffusion system \eqref{eq:main}.

\begin{thm}[Global existence of segregated solutions]\label{thm:existence}
   Let $f$ satisfy Assumption \ref{ass:f conditions}, $V,W \in W^{1,\infty}(\mathbb{R})$, let $\varrho_0,\mu_0 \in L^1(\mathbb{R})\cap \mathrm{Prob}_2(\mathbb{R})$ be given initial data, and assume the finiteness of the energy $ \mathrm{F}[S_0]$, where $S_0:=\varrho_0+\mu_0$. Then, there exists $(\varrho_t,\mu_t)$ a segregated weak solution of \eqref{eq:main} in the sense of Definition \ref{def:segregated}, and the initial data is achieved in the sense $\varrho_t\to \varrho_0$ and $\mu_t\to\mu_0$ in $W_2(\R)$.
\end{thm}

As mentioned, the conclusion of Theorem \ref{thm:existence} follows from a corresponding result for a Lagrangian reformulation (\textit{cf.}~\eqref{eq:main eq pseudo} and \S \ref{sec:lagrangian}), which is encapsulated in our second main theorem. 

\begin{thm}[Lagrangian problem]\label{thm:lagrangian}
   Let $u_0 \in X$ satisfy $\partial_y u_0 \geq 0$ in $\mathcal{M}_\loc((0,2))$ and $\int_0^2 \tilde{f}(u_0') \d y <\infty.$ Let $b$ be given by \eqref{eq:b def} with $V,W \in W^{1,\infty}(\mathbb{R})$ and $\leb(A)=1$. Let $f$ satisfy Assumption \ref{ass:f conditions}. Then, there exists $u \in L^\infty_\loc([0,\infty);X) \cap H^1_\loc([0,\infty);L^2([0,2]))$ satisfying $0\leq - \tilde{f}'(u') \in L^2_\loc([0,\infty);H^1_0([0,2]))$, and 
   \begin{equation}\label{eq:final weak form continuous time in thm}
  \left\lbrace \begin{aligned} 
     &\partial_t u = \partial_y(\tilde{f}'(u')) - b(y,u) \qquad \leb\text{-a.e.~in } (0,\infty)\times[0,2], \\ 
     &u|_{t=0} = u_0, 
      \end{aligned} \right. 
   \end{equation}
where $\lim_{t \to 0^+}\Vert u(t,\cdot) - u_0\Vert_{L^2([0,2])} = 0$, $u'$ as per \eqref{eq:shorthand prime}, and $\essinf_{[0,2]} u'(t,\cdot) > 0$ for $\leb$-a.e.~$t$. Moreover, for $\leb$-a.e.~$t$ the the measure $(\partial_y u)_s$ is concentrated on the set where the function $\tilde{f}'(u')$ vanishes.
\end{thm}

We remark that only the absolutely continuous part $u'$ of the derivative $\partial_y u$ appears in the equation \eqref{eq:final weak form continuous time in thm}. Uniqueness for \eqref{eq:final weak form continuous time in thm} can be obtained under certain semiconvexity conditions on $b$; this lies outside the scope of this paper, as we do not require it to solve \eqref{eq:main}.

\begin{rem}[Comments on the main results]
In this paper we study the cross-diffusion system \eqref{eq:main} on the whole real line $\mathbb{R}$, but it would have been possible to instead consider a bounded interval $[a,b]$, with no-flux boundary conditions. This would have simplified some parts of the analysis, in particular since we would consider (in the Lagrangian reformulation) non-decreasing functions valued into $[a,b]$, which are automatically globally $BV$. This would have avoided the use of locally finite measures and $BV_{\loc}$ functions, and our functional setting would have been endowed with strong compactness in $L^2$ (and not only $L^2_{\loc}$), \textit{cf.}~the topology on the space $X$ of Definition \ref{def:space X}. The restriction $\alpha>\frac13$ in the assumption \eqref{eq:the one third condition} is also related to bounding the second moment of the mass distribution $S=\varrho+\mu$, or equivalently the $L^2$ norm of $u$, and could be removed when working on 
$[a,b]$. On the other hand, the extra difficulty in the bounded domain case lies in the boundary conditions in the Lagrangian formulation, where it would not be possible to prove the Dirichlet boundary values $\tilde f'(u')|_{y=0}=\tilde{f}'(u')|_{y=2}=0$. 
\end{rem}

\section{Lagrangian Reformulation}\label{sec:lagrangian}

In this section we rewrite the system \eqref{eq:main} in Lagrangian form, by arguing via $u_t$ the inverse function of the cumulative density function $F_t$ introduced in \eqref{eq:cdf}. Section \S \ref{sec:formal lagrangian} contains heuristic arguments to give context to the reader before presenting the rigorous definitions (\S\ref{sec:formal gf}--\S\ref{sec:properties of E}) needed for the analysis of the Lagrangian problem in \S \ref{sec:minimising movement}--\S \ref{sec:solution of lagrangian}.

\subsection{Formal Lagrangian reformulation}\label{sec:formal lagrangian}

Recall the cumulative mass function $F_t$ for the sum $S_t$ as defined in \eqref{eq:cdf}. Let us assume for the time being that $F_t$ is strictly increasing, such that it admits a well-defined inverse function $u_t:[0,2] \to \mathbb{R}$ which is also strictly increasing. Our objective is to rewrite the equation \eqref{eq:sum eq} as an equation for $u_t$ \textit{i.e.~}\eqref{eq:main eq pseudo}; this is what we mean by \emph{Lagrangian reformulation}. All manipulations in this portion of the manuscript are formal, and for the purposes of exposition; in what follows, the Eulerian coordinate $x$ and the Lagrangian coordinate $y$ are linked via $y=F_t(x)$ and $x=u_t(y)$.

Notice that \eqref{eq:cdf} formally implies $\partial_x F_t(x) = S_t(x)$, which, using also 
\begin{equation}\label{eq:one sided inverse pseudo}
F_t(u_t(y)) = y, 
\end{equation}
implies, by differentiating with respect to $y$ and using the formula \eqref{eq:cdf}, 
\begin{equation}\label{eq: deriv F in terms deriv G}
   S_t(u_t(y)) = \partial_x F_t(u_t(y)) = \frac{1}{\partial_y u_t(y)}, 
\end{equation}
and taking a further derivative with respect to $y$ yields $\partial_x S_t(u_t(y)) = - \frac{\partial_{yy} u_t(y)}{\partial_y u_t(y)^3}$. Using \eqref{eq:cdf}, we take the primitive of \eqref{eq:sum eq} and, using \eqref{eq:how to get b from rho and mu}, we get $\partial_t F_t = S_t \partial_x f'(S_t) + S_t b$ \textit{i.e.},  
\begin{equation}\label{eq:before closed form eqn for Gt}
    \partial_t F_t(u_t(y)) = S_t(u_t(y)) \bigg[ -f''\big( \frac{1}{\partial_y u_t(y)} \big)\frac{\partial_{yy}u_t(y)}{(\partial_y u_t(y))^3} + b(y,u_t(y)) \bigg]. 
\end{equation}
By \eqref{eq:one sided inverse pseudo} and differentiating in $t$, we get $\partial_t F_t(u_t(y)) + \partial_x F_t(u_t(y))\partial_t u_t(y) = 0$. We deduce $\partial_t F_t(u_t(y)) = -S_t(u_t(y))\partial_t u_t(y).$ Substituting into \eqref{eq:before closed form eqn for Gt} and cancelling by $S_t(u_t(y))$ yields the formal non-divergence equation \eqref{eq:main eq pseudo}, \textit{i.e.}, $\partial_t u_t = f''(\frac{1}{\partial_y u_t}) \frac{\partial_{yy} u_t}{(\partial_y u_t)^3} - b(y,u_t)$.

Recalling the definition \eqref{eq:f tilde def} for $\tilde{f}$ and the computations \eqref{eq:f tilde deriv expressions}, this final equation can be recast in the divergence form $\partial_t u_t = \partial_y(\tilde{f}'(\partial_y u_t)) - b(y,u_t)$; as one can see in equation \eqref{eq:final weak form continuous time in thm}, a a distinction between $\partial_y u$ and $u'$ has to be made, and will be made in the next section, where we also rewrite this divergence-form equation as an $L^2$-gradient flow.

\subsection{Gradient flow formulation of the Lagrangian problem}\label{sec:formal gf}

In \S \ref{sec:define functionals}, we define the functionals over the spaces $\mathcal{M}_\loc((0,2))$ and $BV_\loc((0,2))$ needed for our analysis. \S \ref{sec:formal gf after functionals} links this with the underlying $L^2$-gradient flow structure of the equation obtained in \S \ref{sec:formal lagrangian}. 

\subsubsection{Defining the functionals}\label{sec:define functionals}

For what follows, we recall the function $\tilde{f}$ in \eqref{eq:f tilde def}. This function is now used to define an entropy-like functional.

\begin{definition}[Functionals $\mathrm{Ent}$ and $\mathscr F$]\label{def:functional E}
    For all $\nu \in \mathcal{M}_{\loc}((0,2))$, with $\nu_a$ as per \eqref{eq:absolutely cts and singular decomp}, we define 
    \begin{equation}\label{eq:tilde f def meas}
\mathrm{Ent}[\nu] := \left\lbrace \begin{aligned}
        & + \infty \qquad & \text{if } \nu^- \neq 0, \\ 
        & \int_0^2 \tilde{f}(\nu_a(y)) \d y \qquad & \text{if } \nu^- = 0. 
    \end{aligned} \right. 
\end{equation}
We then define the functional $\mathscr{F}[v]:=\mathrm{Ent}[\partial_y v]$ for all $v\in BV_{\loc}((0,2))$.
\end{definition}

Note that the functional $\mathrm{Ent}$ is a local lower semicontinuous functional on measures (in the sense of the theory of lower semicontinuous functionals in the space of measures, see \cite{BouBut,BouVal}), and the fact that the singular part of $\nu$ does not appear in the functional is due to the condition $ \lim_{s \to \infty} \tilde{f}(s )/{s} = \lim_{s \to \infty} f(1/s) = f(0) = 0$ (see, \textit{e.g.}, \cite[Chapter 7]{Filippo}).

Next, we introduce the functional $\mathscr{E}$ corresponding to $\mathrm{E}$ in \eqref{eq:E functional}. 

\begin{definition}[Functional $\mathscr{E}$]\label{def:functional F}
    For all $v \in X$ and $b \in L^\infty([0,2]\!\times\!\mathbb{R}) $, we define 
    \begin{equation}\label{eq:transformed functional}
    \mathscr{E}[v] := \mathscr{F}[v] + \int_0^2 B(y,v(y)) \d y, 
\end{equation}
where $B(y,s) := \int_0^s b(y,s') \d s'$ is a primitive for $b$ in the second variable. 
\end{definition}

\begin{rem}[Control of $B$]\label{rem:B is lipschitz in third variable ish}
Note that the assumption $b \in L^\infty([0,2]\!\times\!\mathbb{R})$---which follows from \eqref{eq:b def} when $\partial_x V,\partial_x W \in L^\infty(\mathbb{R})$---implies 
\begin{equation}\label{eq:B bound}
   |B(y,s_1) - B(y,s_2)| \leq \Vert b \Vert_{L^\infty} |s_1 - s_2| \qquad \forall s_1,s_2 \in \mathbb{R}.
\end{equation}    
\end{rem}

\subsubsection{Formal $L^2$-gradient flow structure}\label{sec:formal gf after functionals}

By performing the usual formal computation $0=\frac{\der}{\der \delta}|_{\delta=0}\mathscr{E}[v+\delta \varphi]$ for $\varphi$ a smooth test function (assuming formally that $\partial_y v + \delta \partial_y \varphi \geq 0$, \textit{cf.}~Lemma \ref{lem:lower bound competitor} and Proposition \ref{prop:discrete EL}), we obtain that the first variation of $\mathscr{E}$ is 
\begin{equation}\label{eq:F func deriv}
    \mathscr{E}'[v] = b(y,v) - \partial_y( \tilde{f}'(v') ), 
\end{equation}
with the notation $v'$ of Definition \ref{def:BVloc}. Note that in this Hilbert setting we prefer to use the notation $\mathscr{E}'$ rather than the notation with $\frac{\delta}{\delta\varrho}$ that we used in the space of measures, but the meaning is still the first variation of the functional.

Using \eqref{eq:f tilde deriv expressions}, equation \eqref{eq:main eq pseudo} is formally equivalent to 
\begin{equation}\label{eq:pseudo inverse gf}
    \partial_t u = -\mathscr{E}'[u], 
\end{equation}
which is (up to identifying singular parts of the derivative) precisely the divergence-form equation written at the end of \S \ref{sec:formal lagrangian}, \textit{cf.}~\eqref{eq:final weak form continuous time in thm}. We shall therefore prove the existence of a solution of this Lagrangian reformulation (\textit{cf.}~Theorem \ref{thm:lagrangian}) by means of a Minimising Movement Scheme with respect to the functional $\mathscr{E}$ (see \S \ref{sec:minimising movement}--\S \ref{sec:solution of lagrangian}).

The next section contain the properties required on the functionals $\mathscr{E}$, $\mathscr{F}$, $\mathrm{Ent}$ for our later analysis; namely lower semicontinuity, convexity, and boundedness from below.

\subsection{Properties of the functional $\mathscr{E}$}\label{sec:properties of E}

We begin with the following lemma concerned with the convexity of $\mathscr{F}$ and its boundedness from below.

\begin{lemma}[Convexity of $\mathscr{F}$ and lower bound]\label{lem:convexity of E}
  The functional $\mathscr{F}$ is convex and there exist positive constants $C_1,C_2$ depending only on $\alpha$ in \eqref{eq:the one third condition} such that 
    \begin{equation}\label{eq:lower bound functional E}
        \mathscr{F}[u] \geq - C_1 \int_0^2 u'(y)^{1-\alpha} \d y \geq - C_2 \Big( 1 + \Vert u \Vert_{L^2([0,2])}  \Big) \qquad \forall u \in X. 
    \end{equation}
    Moreover, $\tilde{f}(u')_- \in L^1([0,2])$ for all $u \in X$, where $\tilde{f}(u')_-$ is the negative part of $\tilde{f}(u')$. 
\end{lemma}

In order to prove the above, we first need two technical lemmas (Lemmas \ref{lem:blow up ibp} and \ref{lem:fenchel lem}). 

\begin{lemma}[$BV_\loc$ and $L^\infty_{loc}$ estimate for $L^2$ monotonic functions]\label{lem:blow up ibp}
    Let $v \in X$ be such that $\partial_y v \geq 0$ in $\mathcal{M}_\loc((0,2))$. Then, for all $[y_0,y_1] \subseteq [0,2]$ we have the pointwise bounds 
    \begin{equation}\label{eq:blow up bound ibp in app lemma}
        -\frac{\Vert v \Vert_{L^2([0,y_0])}}{\sqrt{y_0}} \leq v(y) \leq \frac{\Vert v \Vert_{L^2([y_1,2])}}{\sqrt{2-y_1}} \qquad \leb\text{-a.e.~}y\in (y_0,y_1). 
    \end{equation}
    Consequently, given any compact subset $[y_0,y_1] \subset (0,2)$, we have the $BV_\loc$-estimate$:$ 
    \begin{equation}\label{eq:bv bound in tech lemma appendix}
        \Vert v \Vert_{TV([y_0,y_1])} \leq \Vert v \Vert_{L^2([0,2])}\Big( \frac{1}{\sqrt{2-y_1}} + \frac{1}{\sqrt{y_0}} \Big). 
    \end{equation}
\end{lemma}

 We remark that the assumption $\partial_y v \geq 0$ in $\mathcal{M}_\loc((0,2))$ above implies immediately that the derivative is a locally finite measure, and hence $v \in BV_\loc$; the real content of the result are the estimates \eqref{eq:blow up bound ibp in app lemma}--\eqref{eq:bv bound in tech lemma appendix}. This result is needed to obtain uniform estimates on the minimising sequences in the proof of Proposition \ref{lem:implicit euler time stepping} (when applying the direct method), since the functional $\mathscr{F}$ is \emph{a priori} not coercive. 
 
 \begin{proof}
We begin with the upper bound. Note that $v$ is an increasing function on $(0,2)$. We then distinguish two cases: if $v(y_1)<0$ the desired inequality is automatically satisfied since $v(y)\leq v(y_1)<0$ for all $y<y_1$. If instead $v(y_1)\geq 0$, we use the fact that for $\leb$-a.e.~$y\in (y_1,2)$ we have $v(y) \geq v(y_1)$. Squaring and integrating in $y$ over $(y_1,2)$ yields 
\begin{equation*}
    (2-y_1)v(y_1)^2 \leq \int_{y_1}^2 v(y)^2 \d y= \Vert v \Vert_{L^2([y_1,2])}^2. 
\end{equation*}
This provides the upper bound 
\begin{equation*}
    v(y_1) \leq \frac{\Vert v \Vert_{L^2([y_1,2])}}{\sqrt{2-y_1}} 
\end{equation*}
and, again by monotonicity, the claim. The proof of the lower bound is similar. For the total variation estimate, we just need to use the equality $||v||_{TV([y_0,y_1])}=v(y_1)-v(y_0)$, which is satisfied whenever $v$ is non-decreasing; recall that when evaluating pointwise values of $v \in BV_\loc((0,2))$ we identify it with its precise representative given in Remark \ref{rem:precise rep BV}. 
\end{proof}

 The next result provides a useful lower bound; its proof relies on Fenchel's inequality. 

\begin{lemma}\label{lem:fenchel lem}
 Let $\alpha$ be as in \eqref{eq:the one third condition}, $0< \beta < (1-\alpha)^{-1}-\frac{3}{2}$, and set 
 \begin{equation}\label{eq:h def}
     h(y) := \big( y(2-y) \big)^{\frac{1}{2}+\beta}, 
 \end{equation}
 which satisfies $\int_{0}^{2} h(y)^{-\frac{1-\alpha}{\alpha}} \d y < \infty$ and $h \in H^1_0([0,2])$. Then, there exists a positive constant $C$ depending only on $\alpha,\beta$, such that for all $[y_0,y_1] \subset (0,2)$ it holds for all $u \in X$ 
    \begin{equation}\label{eq:the fenchel bound in lemma statement}
        \int_{y_0}^{y_1} \! \tilde{f}(u')_-\d y \leq   C \bigg( \Vert \partial_y h \Vert_{L^2([y_0,y_1])} \Vert u \Vert_{L^2([y_0,y_1])} + \int_{y_0}^{y_1} \! h(y)^{-\frac{1\!-\!\alpha}{\alpha}} \d y   + \big[ h(y)u(y) \big]^{y_1}_{y_0} \bigg). 
    \end{equation}
\end{lemma}

\begin{proof}
Using the bound \eqref{eq:the one third condition}  and $\tilde{f}(s)=sf(1/s)$, there exists $C>0$ such that 
    \begin{equation}\label{eq:to get functional lower bound with one third bis}
  \tilde{f}(u'(y))_-\leq  C u'(y)^{1-\alpha} \qquad \leb\text{-a.e.~}(0,2). 
\end{equation}
We begin by showing that, for all $[y_0,y_1] \subset (0,2)$, 
\begin{equation}\label{eq:consequence of fenchel bis}
   - \frac{1}{1-\alpha} \int_{y_0}^{y_1}  u'(y)^{1-\alpha} \d y + \frac{\alpha}{1-\alpha} \int_{y_0}^{y_1} h(y)^{-\frac{{1-\alpha}}{\alpha}} \d y \geq -\int_{y_0}^{y_1} u'(y) h(y) \d y. 
\end{equation}
To see this, recall from Fenchel's inequality that, for all $x\in (0,\infty)$ and $z \in (-\infty,0)$, it holds 
\begin{equation*}
    -\frac{1}{1-\alpha}x^{1-\alpha} + \frac{\alpha}{{1-\alpha}}(-z)^{-\frac{{1-\alpha}}{\alpha}} \geq xz, 
\end{equation*}
where we recognise $\frac{\alpha}{{1-\alpha}}(-z)^{-\frac{{1-\alpha}}{\alpha}} = \sup\{xz + \frac{1}{{1-\alpha}}x^{{1-\alpha}} : x \in (0,\infty) \}$ as the convex conjugate of the function $x \mapsto -\frac{1}{{1-\alpha}}x^{1-\alpha}$; the inequality \eqref{eq:consequence of fenchel bis} is then obtained directly from the above by setting $z=-h(y)$ for $h(y) \geq 0$ and $x=u'(y)$, and integrating over the interval $[y_0 , y_1]$. Then, since $h\geq 0$ and $\partial_y u \geq 0$ in $\mathcal{M}_\loc((0,2))$, the final term in \eqref{eq:consequence of fenchel bis} is estimated as 
\begin{equation*}
    \begin{aligned}
       0 \leq \int_{y_0}^{y_1} u'(y) h(y) \d y \leq \int_{y_0}^{y_1} h(y) \, \partial_y u(y) = - \!\int_{y_0}^{y_1} u(y) \partial_y h(y) \d y \!+\! \big[ h(y_1)u(y_1) \!-\! h(y_0) u(y_0) \big] , 
    \end{aligned}
\end{equation*}
whence, by returning to \eqref{eq:consequence of fenchel bis}, it holds 
\begin{equation}\label{eq:before eps goes to zero in fenchel bis}
    \begin{aligned}
          - \frac{1}{1-\alpha} \int_{y_0}^{y_1} u'(y)^{1-\alpha} \d y & + \frac{\alpha}{1-\alpha} \int_{y_0}^{y_1} h(y)^{-\frac{{1-\alpha}}{\alpha}} \d y \\ 
          &\geq - \Vert u \Vert_{L^2([y_0,y_1])} \Vert \partial_y h \Vert_{L^2([y_0,y_1])} \!+\! \big[ h(y_0) u(y_0) \!-\! h(y_1)u(y_1) \big]. 
    \end{aligned}
\end{equation}
Note that we have $0 \leq  \int_{0}^{2} h(y)^{-\frac{{1-\alpha}}{\alpha}} \d y \leq C \int_0^1 y^{-(\frac{1}{2}+\beta)\frac{{1-\alpha}}{\alpha}} \d y <\infty$, as a consequence of the condition on $\beta$, which implies $-(\frac{1}{2}+\beta)\frac{{1-\alpha}}{\alpha} > -1$. Moreover, we have $\partial_y h \in L^2([0,2])$ since $ |\partial_y h|^2 $ behaves as $(y(2-y))^{2\beta-1}$, which is integrable for $\beta>0$. 
\end{proof}

We are ready to prove Lemma \ref{lem:convexity of E}.

\begin{proof}[Proof of Lemma \ref{lem:convexity of E}]
Since the convexity of the functional is straightforward, we are only concerned with the lower bound \eqref{eq:lower bound functional E}. Note that if the condition $\partial_y u \geq 0$ in $\mathcal{M}_\loc((0,2))$ is not satisfied, then $\mathscr{F}[u] = +\infty$ by Definition \ref{def:functional E}, and the inequality is trivially satisfied. Henceforth, we assume $\partial_y u \geq 0$ in $\mathcal{M}_\loc((0,2))$, and we use \eqref{eq:the fenchel bound in lemma statement} with $y_0=\varepsilon$ and $y_1 = 2-\varepsilon$ for an arbitrary $\varepsilon>0$. We obtain 
\begin{equation}\label{eq:before eps goes to zero in fenchel}
    \begin{aligned}
          \int_\varepsilon^{2-\varepsilon} \tilde f'(u')_- \d y \leq  C \bigg(  \! \Vert u \Vert_{L^2([0,2])} \!+\!\!  \int_{\varepsilon}^{2-\varepsilon} \!\! \!\! \!\!\! h(y)^{-\frac{{1-\alpha}}{\alpha}} \d y  \!+\! h(2\!-\!\varepsilon)u(2\!-\!\varepsilon) \!-\! h(\varepsilon) u(\varepsilon) \! \bigg). 
    \end{aligned}
\end{equation}
Since  $\partial_y u \geq 0$ in $\mathcal{M}_\loc((0,2))$, the assumptions of Lemma \ref{lem:blow up ibp} hold, thus $|u(y)| \leq C ( \frac{1}{\sqrt{y}} + \frac{1}{\sqrt{2-y}} )$ for some $C>0$ depending only on $\Vert u \Vert_{L^2([0,2])}$. This implies $h(\varepsilon)u(\varepsilon)\to 0$ as $\varepsilon \to 0$ and, then
\begin{equation}\label{eq:how to control u' to the 1-alpha}
    \begin{aligned}
          \int_0^{2} \tilde f'(u')_-\d y \leq C \Big( 1 + \Vert u \Vert_{L^2([0,2])} \Big)<+\infty. 
    \end{aligned}
\end{equation}
This allows to prove at the same time \eqref{eq:lower bound functional E} and $\tilde{f}(u')_- \in L^1([0,2])$. 
\end{proof}

\begin{lemma}[Properties of $\mathscr{E}$]\label{lem:functional prop}
    The functional $\mathscr{E}: X \to \mathbb{R}$ of Definition \ref{def:functional F} is proper and 
\begin{equation}\label{eq:func lower bound for W11}
    \mathscr{E}[v] \geq - C \Big(1 + \Vert v \Vert_{L^2([0,2])} \Big)  \qquad \text{for all } v \in X, 
\end{equation} 
where the positive constant $C$ depends only on $\Vert b \Vert_{L^\infty}$ and the exponent $\alpha$ in \eqref{eq:the one third condition}. Furthermore, if $\{v_n\}_{n}$ is a sequence converging to $v$ in $X$, then 
$\mathscr{E}[v] \leq \liminf_{n\to\infty}\mathscr{E}[v_n]$ and 
\begin{equation}\label{eq:lsc Fb func}
 \int_{y_0}^{y_1} \tilde{f}(v') \d y \leq \liminf_{n\to\infty} \int_{y_0}^{y_1} \tilde{f}(v_n') \d y \qquad \forall [y_0,y_1] \subseteq [0,2]. 
\end{equation}
\end{lemma}

\begin{proof}
\noindent 1. \textit{Estimate and proper functional}. The estimate \eqref{eq:func lower bound for W11} follows immediately from the lower bound \eqref{eq:lower bound functional E} and the estimate \eqref{eq:B bound} on the drift term along with an application of H\"older's inequality. It is clear that $\mathscr{E}$ is proper since, for instance, the linear function $v_*(y) = y$ for all $y\in [0,2]$ belongs to the space $X$, and it holds 
\begin{equation*}
    \mathscr{E}[v_*] = \int_0^2 f(1) \d y + \int_0^2 B(y,y) \d y \leq 2 \big( f(1) + \Vert b \Vert_{L^\infty} \big) < \infty. 
\end{equation*}

\smallskip 
\noindent 2. \textit{Lower semicontinuity}. Note that $B$ is continuous in its second argument and has at most linear growth. From the strong convergence $v_n \to v$ in $L^1([0,2])$ we deduce
\begin{equation}\label{eq:converg B vn to B v using Helly}
   \int_0^2 B(y,v) \d y = \lim_{n\to\infty} \int_0^2 B(y,v_n) \d y. 
\end{equation}

Denoting $v', v_n' \in L^1_\loc((0,2))$ the absolutely continuous parts of the derivatives as in \eqref{eq:shorthand prime}, we obtain the lower semicontinuity: 
\begin{equation}\label{eq:before the compact exhaustion}
  \int_{y_0}^{y_1} \tilde{f}(v') \d y \leq \liminf_{n\to\infty} \int_{y_0}^{y_1} \tilde{f}(v_n') \d y, 
\end{equation}
on any compact subset $[y_0,y_1] \subset (0,2)$ since the $BV_\loc$ bound on $v_n$ implies weak-* convergence of $\partial_y v_n$ to $\partial_y v$ on $[y_0,y_1]$ and the functional $\mathrm{Ent}$ is lower semicontinuous for this convergence. 

We have therefore verified \eqref{eq:lsc Fb func} for all $[y_0,y_1] \subset (0,2)$, and it remains to verify it for $[y_0,y_1]=[0,2]$. Using \eqref{eq:the fenchel bound in lemma statement} with lower endpoint $0$ and upper endpoint $y_0 \in (0, 1)$, noting that $v(0)h(0) = 0$ with $h$ as prescribed by \eqref{eq:h def}, we get 
\begin{equation}\label{eq:y0 close to 0}
    \int_0^{y_0} \tilde{f}(v_n') \d y \geq - C \bigg( \Vert \partial_y h \Vert_{L^2([0,y_0])} + \int_{0}^{y_0} h(y)^{-\frac{1-\alpha}{\alpha}} \d y   +  \frac{1}{\sqrt{y_0}}h(y_0)  \bigg), 
\end{equation}
where we also used that $v_n(y_0) \geq -\frac{\sup_n \Vert v_n \Vert_{L^2([0,2])}}{\sqrt{y_0}}$ for $y_0 \in (0,1)$ from Lemma \ref{lem:blow up ibp}. Analogously, for $1<y_1<2$, since $u(2)h(2)=0$, we have 
\begin{equation}\label{eq:y1 close to 2}
    \int_{y_1}^{2} \tilde{f}(v_n') \d y \geq - C \bigg( \Vert \partial_y h \Vert_{L^2([y_1,2])} + \int_{y_1}^{2} h(y)^{-\frac{1-\alpha}{\alpha}} \d y   +  \frac{1}{\sqrt{2-y_1}}h(y_1)  \bigg), 
\end{equation}
where we also used that $v_n(y_1) \leq \frac{\sup_n \Vert v_n \Vert_{L^2([0,2])}}{\sqrt{2-y_1}}$ for $y_1 \in (1,2)$ from Lemma \ref{lem:blow up ibp}. As such, 
\begin{equation*}
 \begin{aligned}
      \int_0^{2} \tilde{f}(v_n') \d y =& \int_0^{y_0} \tilde{f}(v_n') \d y + \int_{y_0}^{y_1} \tilde{f}(v_n') \d y + \int_{y_1}^{2} \tilde{f}(v_n') \d y \\ 
      \geq & \int_{y_0}^{y_1} \tilde{f}(v_n') \d y - C \bigg( \Vert \partial_y h \Vert_{L^2([0,y_0])} + \int_{0}^{y_0} h(y)^{-\frac{1-\alpha}{\alpha}} \d y   +  \frac{1}{\sqrt{y_0}}h(y_0)  \bigg)\\ 
      &- C \bigg( \Vert \partial_y h \Vert_{L^2([y_1,2])} + \int_{y_1}^{2} h(y)^{-\frac{1-\alpha}{\alpha}} \d y   +  \frac{1}{\sqrt{2-y_1}}h(y_1)  \bigg), 
 \end{aligned}
\end{equation*}
for all $0 < y_0 < 1 < y_1 < 2$. By taking the $\liminf_n$ on both sides and using \eqref{eq:before the compact exhaustion}, we get 
\begin{equation}\label{eq:final part of fenchel proof}
 \begin{aligned}
     \liminf_{n\to\infty} \int_{0}^{2} \tilde{f}(v_n') \d y 
 \geq & \int_{y_0}^{y_1} \tilde{f}(v') \d y \!-\! C \bigg( \Vert \partial_y h \Vert_{L^2([0,y_0])} \!+ \!\int_{0}^{y_0}\!\! h(y)^{-\frac{1-\alpha}{\alpha}} \d y   \!+\!  \frac{1}{\sqrt{y_0}}h(y_0)  \bigg)\\ 
      &- C \bigg( \Vert \partial_y h \Vert_{L^2([y_1,2])} + \int_{y_1}^{2} h(y)^{-\frac{1-\alpha}{\alpha}} \d y   +  \frac{1}{\sqrt{2-y_1}}h(y_1)  \bigg). 
 \end{aligned}
\end{equation}
We let $y_0 \to 0^+$ and $y_1 \to 2^-$ in \eqref{eq:final part of fenchel proof}. Recall from the final part of Lemma \ref{lem:convexity of E} that the negative part $\tilde{f}(v')_-$ is integrable, whence $\int_{y_0}^{y_1} \tilde{f}(v') \d y $ tends to $ \int_0^2 \tilde{f}(v') \d y$ as $y_0\to 0$ and $y_1\to 1$; note that this integral may assume the value $+\infty$, but it is bounded from below by Lemma \ref{lem:convexity of E}. Hence, using the integrability of $h$ and $\partial_y h$ from Lemma \ref{lem:fenchel lem}, we get $\liminf_{n} \int_0^2 \tilde{f}(v_n') \d y \geq \int_0^2 \tilde{f}(v') \d y$. Using \eqref{eq:converg B vn to B v using Helly}, the proof of \eqref{eq:lsc Fb func} is complete. 
\end{proof}

We now record the monotonicity properties of $\tilde{f}'$ which are key to passing to the limit in the time-step $\tau \to 0$ in \S \ref{sec:solution of lagrangian} to prove existence to the Lagrangian reformulation. 

\begin{lemma}[Quantified strict convexity of $\tilde{f}$]\label{rem:convexity and monotonicity of tilde f}
Recall the function $\tilde{f}$ defined in \eqref{eq:f tilde def}. 
Then there exists a non-negative continuous function $\omega : [0,\infty)\times(0,\infty) \to [0,\infty]$ satisfying $\omega(s_1,s_2)=0$ if and only if $s_1=s_2$, and 
    \begin{equation}\label{eq:convexity for tilde f second gives L2 bound}
        \tilde{f}(s_1) - \tilde{f}(s_2) - (s_1 - s_2) \tilde{f}'(s_2) \geq \omega(s_1,s_2) \qquad \forall s_1,s_2 \in (0,\infty). 
    \end{equation}

  Furthermore, for all fixed $s_1 \in (0,\infty)$, $\omega$ admits the limits 
    \begin{equation}\label{eq:bad limits for omega}
       0 < \lim_{s_2 \to 0} \omega(s_1,s_2) \leq \infty, \qquad  0 < \lim_{s_2 \to \infty} \omega(s_1,s_2) \leq \infty. 
    \end{equation}
\end{lemma}

\begin{proof}
    Let $s_1,s_2 \in (0,\infty)$. We define the quantity $\omega$ by 
    \begin{equation}\label{eq:meet omega}
        \tilde{f}(s_1) - \tilde{f}(s_2) - (s_1 - s_2) \tilde{f}'(s_2) = \int^{s_1}_{s_2} \int_{s_2}^r \tilde{f}''(s) \d s \d r  =: \omega(s_1,s_2).
    \end{equation}
   This formula guarantees the continuity of $\omega$ on $[0,\infty)\times (0,\infty)$ using the integrability of $\tilde{f}''$ on all compact subsets of $(0,\infty)$, and shows that $\omega$ vanishes on the diagonal. Also, this expression clearly provides $\omega(s_1,s_2) > 0$ for all $s_1 \neq s_2$ because of the strict convexity of $\tilde{f}$, which implies that $\tilde{f}''$ does not vanish on any interval; hence there exists a point in the interval $(s_2,r)$ where $\tilde{f}''$ is strictly positive, and by continuity of $\tilde{f}''$ a small neighbourhood on which this holds, giving rise to a strictly positive integral. Similarly, we obtain for the same reason  $ \lim_{s_2 \to 0} \omega(s_1,s_2)>0$ and $ \lim_{s_2 \to \infty} \omega(s_1,s_2) > 0$. %Then we manifestly have $\omega(s,s)=0$.
\end{proof}

\section{Minimising Movement Scheme for the Lagrangian Problem}\label{sec:minimising movement}

In order to show existence to \eqref{eq:final weak form continuous time in thm}, we proceed by means of an implicit Euler (Minimising Movement) scheme discretised with respect to the time variable. Our result is the following. 

\begin{prop}[Implicit Euler time-stepping scheme]\label{lem:implicit euler time stepping}
    Let $u_0 \in X$ be such that $\mathscr{F}[u_0]$ is finite, $T>0$ be arbitrary, $ \tau > 0$, and $N := \ceil{T/\tau}$. For $k \in \{1,\dots,N\}$, define the functionals 
    \begin{equation}\label{eq:def of functional Fk}
        \mathscr{E}^k[v] := \frac{1}{2\tau}\Vert v - v_{k-1} \Vert_{L^2([0,2])}^2 + \mathscr{E}[v], 
    \end{equation}
    with $v_0 := u_0$, and $v_k \in \argmin \big\{ \mathscr{E}^k[v] : v \in X \big\}$. Then, the minimisers $\{v_k \}_{k=1}^N \subset X$ are well-defined, and satisfy $\partial_y v_k \geq 0$ in $\mathcal{M}_\loc((0,2))$ and $\int_0^2 \tilde{f}(v_k') \d y<+\infty$. 
\end{prop}

\begin{proof} We emphasise at the start of the proof that $\tau$ and $k$ are fixed throughout. By the estimate \eqref{eq:lower bound functional E}, the functional $\mathscr{E}^k$ satisfies 
\begin{equation*}
        \mathscr{E}^k[v] \geq \frac{1}{2\tau} \Vert v - v_{k-1} \Vert^2_{L^2([0,2])} - 2C \big( 1+\Vert v \Vert_{L^2([0,2])} \big).\end{equation*}
It follows that $\mathscr{E}^k$ is bounded from below and coercive in terms of the $L^2$ norm.

We now argue by the direct method of the calculus of variations. By definition of the infimum, there exists a minimising sequence $\{v_n\}_{n\in\mathbb{N}} \subset X$: 
\begin{equation*}
 \mathscr{E}^k[v_n] <+\infty,\quad   \lim_{n\to\infty} \mathscr{E}^k[v_n] = \inf\big\{ \mathscr{E}^k[w] : w \in X  \big\} =: M^k > -\infty. 
\end{equation*}
The aforementioned coercivity implies that the sequence $v_n$ is bounded in $L^2([0,2])$. Recall also $\partial_y v_n\geq 0$ for all $n$, and these two facts together imply (by \eqref{eq:bv bound in tech lemma appendix}) that $v_n$ is bounded in $X$, and hence compact for the convergence in $X$. The lower semicontinuity result of Lemma \ref{lem:functional prop} and the weak lower semicontinuity of the $L^2$ norm imply the existence of a minimiser $v \in X$ for the functional $\mathscr{E}^k$, which is such that $\int_0^2 \tilde{f}(v') \d y$ is finite and satisfies $(\partial_y v)_s \geq 0$ in the sense $\mathcal{M}_\loc((0,2))$ and $v' \geq 0$ $\leb$-a.e.~in [0,2]. We set $v_k = v$. 
\end{proof}

In fact, a stronger lower bound is available than simply $\partial_y v_k \geq 0$ in $\mathcal{M}_\loc((0,2))$ for minimisers $v_k$ of $\mathscr{E}^k$. This is encapsulated in the next lemma, and is fundamental in obtaining the Euler--Lagrange equation associated to $\mathscr{E}^k$, \textit{cf.}~Proposition \ref{prop:discrete EL}. 

\begin{lemma}[Strictly positive lower bound on $v_k'$]\label{lem:lower bound competitor}
    Assume the conditions of Proposition \ref{lem:implicit euler time stepping} hold, and let $\mathscr{E}^k$ be as per \eqref{eq:def of functional Fk} and $v_k \in \argmin\{\mathscr{E}^k[v]:v\in X\}$. Then, it holds 
    \begin{equation}\label{eq:lower bound deriv minimiser}
        \essinf_{[0,2]} v_k' > 0, 
    \end{equation}
    where we use the notation $v_k'$ of Definition \ref{def:BVloc}. 
\end{lemma}

\begin{proof}
    For ease of notation we omit the subscript $k$ from the minimiser $v_k$ of $\mathscr{E}^k$ in this proof. Suppose for contradiction that \eqref{eq:lower bound deriv minimiser} does not hold, \textit{i.e.}, we suppose 
    \begin{equation}\label{eq:essinf is zero for contradiction}
        \essinf_{[0,2]} v' = 0. 
    \end{equation}
    Our strategy is now to construct a competitor $\tilde{v} \in X$ such that $\mathscr{E}^k[\tilde{v}] < \mathscr{E}^k[v]$, which contradicts the minimality of $v$ in the class $X$; we explain the underlying idea in the lines that follow. By \eqref{eq:f tilde prime blows up at zero assumption-1}, $\tilde{f}'(v')$ explodes as $v'$ approaches the value zero and, by the contradiction hypothesis \eqref{eq:essinf is zero for contradiction}, $v'$ takes arbitrarily small values on sets of positive measure. We therefore build a competitor $\tilde{v}$, modifiying some small values of $v'$ into larger values, so that $\int_0^2 \tilde{f}(\tilde{v}') \d y$ is much less than $\int_0^2 \tilde{f}(v') \d y$ and such that the gain in this term overcomes the possible loss in the other terms of the functional $\mathscr{E}^k$.
    \smallskip 

    \noindent 1. \textit{Construction of a competitor}. For all $\varepsilon>0$, we define the competitor 
    \begin{equation}\label{eq:competitor}
        \tilde{v}(y) := v(y) + \int_{0}^y (2\varepsilon-v'(z)) \mathds{1}_{\{0 \leq v'(z) < \varepsilon\}} \d z \qquad \mathscr{L}\text{-a.e.~}y \in [0,2]; 
    \end{equation}
    we omit the explicit dependence on $\varepsilon$ to avoid heavy notation. Since the integrand is bounded by $2\varepsilon$, the integral is well-defined.

    Moreover, $\tilde{v} \in X$ for all choices of $\varepsilon > 0$ and, by differentiating \eqref{eq:competitor} distributionally, 
\begin{equation*}
    \partial_y \tilde{v} = \partial_y v + (2\varepsilon- v') \mathds{1}_{\{0 \leq v' < \varepsilon\}} \geq 0, 
\end{equation*}
where the final inequality holds in the sense $\mathcal{M}_\loc((0,2))$, and follows directly from $\partial_y v \geq 0$. By taking the absolutely continous part of the measure above, which is a linear operation, 
\begin{equation*}
   \begin{aligned}
       \tilde{v}' = v' + (2\varepsilon- v') \mathds{1}_{\{0 < v' < \varepsilon\}} = v' \mathds{1}_{\{v' \geq \varepsilon\}} + 2\varepsilon \mathds{1}_{\{0 \leq v' < \varepsilon\}} \qquad \leb\text{-a.e.~in } [0,2], 
   \end{aligned} 
\end{equation*}
from which it follows from Definition \ref{def:functional E} of the functional $\mathscr{F}$ that 
\begin{equation}\label{eq:E func value for competitor}
    \mathscr{F}[\tilde{v}] = \int_0^2 \tilde{f}(v')  \mathds{1}_{\{v' \geq \varepsilon\}} \d y + \tilde{f}(2\varepsilon) m(\varepsilon), 
\end{equation}
where we have used the shorthand notation $m(\varepsilon) := \mathscr{L}(\{0 \leq v' < \varepsilon\})$. Note that the contradiction hypothesis \eqref{eq:essinf is zero for contradiction} implies $
    m(\varepsilon) > 0$ for all $\varepsilon>0.$

\smallskip 

\noindent 2. \textit{Estimates on competitor}. We compute the functional value for the competitor $\tilde{v}$, 
\begin{equation*}
    \begin{aligned}
         \mathscr{E}^k[\tilde{v}]  - \mathscr{E}^k[v] = & \frac{1}{2\tau} \Vert \tilde{v} - v_{k-1} \Vert^2_{L^2([0,2])} - \frac{1}{2\tau} \Vert v - v_{k-1} \Vert^2_{L^2([0,2])} \\ 
         &+ \int_0^2 \Big( B(y,\tilde{v}(y)) \!-\! B(y,v(y)) \Big) \d y \!+\! \mathscr{F}[\tilde{v}] - \mathscr{F}[v]. 
    \end{aligned}
\end{equation*}
By expanding the $L^2$-inner product, denoted $\langle\cdot,\cdot\rangle$, and using the Lipschitz bound \eqref{eq:B bound} for $B$, 
\begin{equation*}
    \begin{aligned}
        \bigg| \mathscr{E}^k[\tilde{v}] & \! - \!\mathscr{E}^k[v] \!-\! ( \mathscr{F}[\tilde{v}] - \mathscr{F}[v]) \bigg| \\ 
         \leq & \, \frac{1}{2\tau}\Big| \Vert \tilde{v} \Vert^2_{L^2([0,2])} \!-\! \Vert v \Vert^2_{L^2([0,2])}  \!-\! 2 \langle \tilde{v} \!-\! v , v_{k-1} \rangle \Big| + {\Vert b \Vert_{L^\infty}} \Vert \tilde{v} - v \Vert_{L^1([0,2])}. 
    \end{aligned}
\end{equation*}
Returning to \eqref{eq:competitor}, we note the pointwise estimate $|\tilde{v}(y) - v(y)| \leq 2\varepsilon m(\varepsilon) $ for $\mathscr{L}$-a.e.~$y$, and 
\begin{equation*}
   \begin{aligned}
       \Vert \tilde{v} \Vert^2_{L^2} = \Vert v \Vert^2_{L^2} + 2 \langle v , \tilde{v} - v \rangle +  \Vert \tilde{v} - v \Vert^2_{L^2([0,2])}, 
   \end{aligned} 
\end{equation*}
from which we get $| \Vert \tilde{v} \Vert^2_{L^2([0,2])}- \Vert v \Vert^2_{L^2([0,2])} | \leq 4\sqrt{2} \Vert v\Vert_{L^2([0,2])} \varepsilon m(\varepsilon) +  8 \varepsilon^2 m(\varepsilon)^2$. In turn, 
\begin{equation*}
    \begin{aligned}
        \bigg| \mathscr{E}^k[\tilde{v}] \! - \!\mathscr{E}^k[v] \!-\! ( \mathscr{F}[ \tilde{v}] - \mathscr{F}[ v]) \bigg|         \leq  C \varepsilon m(\varepsilon) \big( 1 + \varepsilon  m(\varepsilon)  \big), 
    \end{aligned}
\end{equation*}
where the positive constant $C$ depends on $\Vert v \Vert_{L^2}, \Vert v_{k-1} \Vert_{L^2}, \Vert b \Vert_{L^\infty}$, $\tau$, but is independent of $\varepsilon$. Meanwhile, by \eqref{eq:E func value for competitor}, we also have $\mathscr{F}[ \tilde{v}] - \mathscr{F}[ v] = -\int_0^2 \tilde{f}(v') \mathds{1}_{\{0 \leq v' < \varepsilon\}} \d y + \tilde{f}(2\varepsilon) m(\varepsilon)$ so that the previous inequality implies the upper bound 
\begin{equation}\label{eq:almost there competitor}
    \begin{aligned}
        \mathscr{E}^k[\tilde{v}] &\leq \mathscr{E}^k[v] -\! \int_0^2 \tilde{f}(v') \mathds{1}_{\{0 \leq v' < \varepsilon\}} \d y + \tilde{f}(2\varepsilon) m(\varepsilon) + C \varepsilon m(\varepsilon) \big( 1 + \varepsilon  m(\varepsilon)  \big); 
    \end{aligned}
\end{equation}
the above holds for all choices of $\varepsilon > 0$, with $C$ independent of $\varepsilon$.

\smallskip 

\noindent 3. \textit{Competitor beats minimiser}. Our aim is to show that the contribution from the negative integral term in \eqref{eq:almost there competitor} dominates, so that we have the strict inequality $\mathscr{E}^k[\tilde{v}] < \mathscr{E}^k[v]$. 

As $\tilde{f}$ is strictly decreasing (Remark \ref{properties of tilde f}), it holds $\int_0^2 \tilde{f}(v') \mathds{1}_{\{0 \leq v' < \varepsilon\}} \d y \geq \tilde{f}(\varepsilon) m(\varepsilon)$, whence 
\begin{equation}\label{eq:really near the end of competitor proof}
    \begin{aligned}
        \mathscr{E}^k[\tilde{v}] &\leq \mathscr{E}^k[v] \!+\! m(\varepsilon) \Big[ - \tilde{f}(\varepsilon) + \tilde{f}(2\varepsilon)  + C \varepsilon \big( 1 + \varepsilon m(\varepsilon)  \big) \Big] \qquad \text{for all } \varepsilon>0, 
    \end{aligned}
\end{equation}
and because of the hypothesis \eqref{eq:essinf is zero for contradiction}, it suffices to show that the term in the square brackets is strictly negative for some choice of $\varepsilon$; indeed, recall that the quantity $m(\varepsilon)$ is strictly positive for all choices of $\varepsilon>0$. This quantity is also bounded above by $m(\varepsilon) \leq \leb([0,2]) = 2$. It follows that, for all $\varepsilon \in (0,1)$ we have $1+\varepsilon m(\varepsilon) \leq 3$, and \eqref{eq:really near the end of competitor proof} yields 
\begin{equation}\label{eq:penultimate competitor proof}
    \begin{aligned}
        \mathscr{E}^k[\tilde{v}] &< \mathscr{E}^k[v] + \varepsilon m(\varepsilon) \Big[ \frac{\tilde{f}(2\varepsilon) - \tilde{f}(\varepsilon)  }{\varepsilon}  + 3C \Big] \qquad \forall \varepsilon \in (0,1). 
    \end{aligned}
\end{equation}
Meanwhile, the convexity of $\tilde f$ gives 
$
       \frac{\tilde{f}(2\varepsilon) - \tilde{f}(\varepsilon)  }{\varepsilon}\leq \tilde{f}'(2\varepsilon) \to -\infty$ as $\varepsilon \to 0^+$ by \eqref{eq:f tilde prime blows up at zero assumption-1}. Thus, by choosing $\varepsilon$ sufficiently small, we contradict the minimality of $v$ in $X$. 
\end{proof}

The strictly positive lower bound on $v_k'$ from the previous lemma allows us to compute the Euler--Lagrange equation satisfied by the minimiser $v_k$; this is the content of the next result, which relies crucially on the aforementioned strictly positive lower bound on the derivative.

\begin{prop}[Discrete-time Euler--Lagrange]\label{prop:discrete EL}
        The minimisers $\{v_k \}_{k=1}^N \subset X$ of Proposition \ref{lem:implicit euler time stepping} satisfy $\essinf_{[0,2]} v_k' > 0$, $(\partial_y v_k)_s \geq 0$ in $\mathcal{M}_\loc((0,2))$, $\tilde{f}'(v_k') \in H^1([0,2])$, and 
\begin{equation}\label{eq:time discrete}
    \frac{v_k - v_{k-1}}{\tau} = \partial_y(\tilde{f}'(v_k')) - b(y,v_k), 
\end{equation}
   %weakly in duality with $H^1([0,2])$, \textit{i.e.}, for all $\varphi \in H^1([0,2])$, 
   %\begin{equation}\label{eq:discrete EL in prop}
    %   \int_0^2 \bigg( \varphi \Big(\frac{v_k - v_{k-1}}{\tau} + b(y,v_k) \Big) + \tilde{f}'(v_k') \partial_y \varphi \bigg) \d y = 0, 
  % \end{equation}
    with $v_k'$ as per \eqref{eq:shorthand prime}. For some positive universal constant $C$ we also have
    \begin{equation}\label{eq:L2 gradient of bad deriv discrete}
    \begin{aligned}
        \Vert \tilde{f}'(v_k') \Vert_{H^1([0,2])} &\leq C \Big( \big\Vert \frac{v_k-v_{k-1}}{\tau} \big\Vert_{L^2([0,2])}  + \Vert b \Vert_{L^\infty} \Big). 
    \end{aligned}
    \end{equation}
\end{prop}

\begin{proof}
The proof is divided into three steps. 

\smallskip 

\noindent 1. \textit{Euler--Lagrange equation in duality with $C^1([0,2])$}. Let $\varphi \in C^1([0,2])$. Recall from Lemma \ref{lem:lower bound competitor} that it holds $\ell := \essinf_{[0,2]} v_k' >0$. In turn, for all $\delta \in \mathbb{R}$ satisfying 
    \begin{equation}\label{eq:for all such delta}
        |\delta| < \delta_* := \frac{\ell}{(1+\Vert \partial_y \varphi \Vert_{L^\infty})}, 
        \end{equation}
    it holds, using the notation \eqref{eq:shorthand prime}, $\partial_y (v_k + \delta \varphi) = \partial_y v_k + \delta \partial_y \varphi \geq \ell_\varphi \!\cdot\! \mathscr{L} > 0$, where $\ell_\varphi := \ell( 1 - \frac{\Vert \partial_y \varphi \Vert_{L^\infty}}{1+\Vert \partial_y \varphi \Vert_{L^\infty}} )$.   Thus, for all $\varphi \in C^1([0,2])$ and $\delta$ satisfying \eqref{eq:for all such delta}, it holds 
    \begin{equation*}
        \mathscr{E}^k[v_k+\delta \varphi] = \frac{1}{2\tau}\Vert v_k + \delta \varphi - v_{k-1} \Vert_{L^2([0,2])}^2+ \int_0^2 \tilde{f}(v_k' + \delta \partial_y \varphi) \d y + \int_0^2 B(y,v_k+\delta \varphi) \d y. 
    \end{equation*}
The minimality of $v_k$ in the class $X \ni v_k + \delta \varphi$ implies $\mathscr{E}^k[v_k] \leq \mathscr{E}^k[v_k+\delta \varphi]$, whence 
\begin{equation}\label{eq:minimality to get discrete EL}
     \begin{aligned}  0 \leq \int_0^2 \Big( \tilde{f}(v_k' + \delta \partial_y \varphi) - \tilde{f}(v_k' ) \Big) \d y &+ \int_0^2 \Big( B(y,v_k +\delta \varphi) - B(y,v_k) \Big) \d y \\ 
     &+ \frac{1}{2\tau}\int_0^2 \Big( |v_k+ \delta \varphi - v_{k-1} |^2 - |v_k-v_{k-1}|^2  \Big) \d y. 
     \end{aligned}
    \end{equation}
Choosing $\delta \in (0,\delta_*)$ and dividing by $\delta$, and then letting $\delta \to 0^+$, we get 
\begin{equation}\label{eq:just before dct in EL derivation}
       0 \leq \lim_{\delta \to 0^+}\int_0^2 \frac{\tilde{f}(v_k' + \delta \partial_y \varphi) - \tilde{f}(v_k')}{\delta} \d y + \int_0^2 \Big( b(y,v_k) + \frac{v_k-v_{k-1}}{\tau}   \Big) \varphi \d y. 
    \end{equation}
    To evaluate the first term on the right-hand side, we use
        \begin{equation*}
        \Big| \frac{\tilde{f}(v_k' + \delta \partial_y \varphi) - \tilde{f}(v_k')}{\delta} \Big| \leq \sup_{s \geq \ell_\varphi}|\tilde{f}'(s)| \cdot \Vert \partial_y \varphi \Vert_{L^\infty} = -\tilde{f}'(\ell_\varphi) \Vert \partial_y \varphi \Vert_{L^\infty} \qquad \mathscr{L}\text{-a.e.}, 
    \end{equation*}
    where we used that $\tilde{f}' \leq 0$ and $\tilde{f}'$ is increasing (\textit{cf.}~Remark \ref{properties of tilde f}) to obtain the final inequality. In turn, by returning to \eqref{eq:just before dct in EL derivation} and applying the Dominated Convergence Theorem, we get 
    \begin{equation*}
       0 \leq \int_0^2 \tilde{f}'(v_k')\partial_y \varphi \d y + \int_0^2 \Big( b(y,v_k) + \frac{v_k-v_{k-1}}{\tau}   \Big) \varphi \d y \qquad \forall \varphi \in C^1([0,2]). 
    \end{equation*}
    Replacing $\varphi$ with $-\varphi$ we finally obtain 
equality for all $\varphi \in C^1([0,2])$, which proves \eqref{eq:time discrete} in duality with any $\varphi \in C^1([0,2])$. 

    \smallskip 

    \noindent 2. \textit{$L^2$-estimate on $\partial_y (\tilde{f}'(v_k'))$}. Directly from \eqref{eq:time discrete}, we obtain in the sense of distributions $\partial_y (\tilde{f}'(v_k'))=\frac{v_k-v_{k-1}}{\tau}-b$, which shows that $\partial_y (\tilde{f}'(v_k'))$ belongs to $L^2([0,2])$ with 
    \begin{equation}\label{eq:discrete EL to get first L2 est}
        \Vert \partial_y(\tilde{f}'(v_k')) \Vert_{L^2([0,2])} \leq \Big( \big\Vert \frac{v_k-v_{k-1}}{\tau} \big\Vert_{L^2([0,2])}  + \sqrt{2}\Vert b \Vert_{L^\infty} \Big). 
    \end{equation}

    \smallskip 

    \noindent 3. \textit{$L^1$-estimate on $\tilde{f}'(v_k')$}. In order to obtain a full bound in $H^1$ on $\tilde{f}'(v_k')$ we will also derive an $L^1$-estimate on it. To this end, we test the equation with $\varphi(y)=y$, for which $\partial_y \varphi = 1$ (note $\varphi \in C^1([0,2])$), and obtain (using $\tilde{f}' \leq 0$) 
    \begin{equation*}
        \Vert \tilde{f}'(v_k') \Vert_{L^1} = \int_0^2 - \tilde{f}'(v_k') \d y = \int_0^2 y \Big(\frac{v_k - v_{k-1}}{\tau} + b(y,v_k) \Big) \d y. 
    \end{equation*}
    Then, by the Poincar\'e--Wirtinger inequality, we 
  %   \begin{equation*}
   %     \Big\Vert \tilde{f}'(v_k') - \frac{1}{2}\int_0^2 \tilde{f}'(v_k') \d y \Big\Vert_{L^2([0,2])} \leq C_P \Vert \partial_y (\tilde{f}'(v_k')) \Vert_{L^2([0,2])}, 
    %\end{equation*}
    also deduce a bound on the whole $H^1$ norm of $\tilde{f}'(v_k') $. 
   % \begin{equation*}
    %    \Vert \tilde{f}'(v_k') \Vert_{L^2([0,2])} \leq C\big(\Vert \tilde{f}'(v_k') \Vert_{L^1([0,2])} + \Vert \partial_y (\tilde{f}'(v_k')) \Vert_{L^2([0,2])} \big), 
    %\end{equation*}
    %where the constant $C$ depends only on the domain $[0,2]$. The final estimate \eqref{eq:L2 gradient of bad deriv discrete} now follows from the previous $L^2$-bound and \eqref{eq:discrete EL to get first L2 est}. In turn, using the density of $C^1([0,2])$ in $H^1([0,2])$ it follows that the weak form \eqref{eq:discrete EL in prop} holds in duality with any test function $\varphi \in H^1([0,2])$. 
    Now that we know that we have $\tilde{f}'(v_k') \in H^1$, the equation holds in weak form in duality with less smooth functions $\varphi$, and also in strong form as an aquality a.e. with the ($L^2$) distributional derivative of $\tilde{f}'(v_k')$.
\end{proof}

\begin{rem}[Improved weak formulation]\label{rem:improved weak form no time}
    Choosing to test \eqref{eq:time discrete} with $\varphi \in C^1_c((0,2))$, using the $H^1$-bound on $\tilde{f}'(v_k')$ obtained in \eqref{eq:L2 gradient of bad deriv discrete}, and integrating by parts, we deduce 
    \begin{equation}\label{eq:discrete EL in prop duality L2}
       \int_0^2 \varphi \bigg(\frac{v_k - v_{k-1}}{\tau} + b(y,v_k) - \partial_y(\tilde{f}'(v_k')) \bigg) \d y = 0. 
   \end{equation}
   Moreover, due to the density of $C^1_c((0,2))$ in $L^2([0,2])$, the weak formulation \eqref{eq:discrete EL in prop duality L2} holds for all $\varphi \in L^2([0,2])$, and in fact the Fundamental Lemma of the calculus of variations implies 
   \begin{equation}\label{eq:discrete EL in prop duality L2 with fundamental lem of cov}
       \frac{v_k - v_{k-1}}{\tau} = \partial_y(\tilde{f}'(v_k')) - b(y,v_k) \qquad \leb\text{-a.e.~in } [0,2]. 
   \end{equation}
\end{rem}

\begin{cor}[$\tilde{f}'(v_k')$ vanishes at the endpoints]\label{cor:vanish endpoints}
    It holds $\tilde{f}'(v_k') \in H^1_0([0,2])$, \textit{i.e.}~
    \begin{equation}\label{eq:vanish at end points}
    \tilde{f}'(v_k')|_{y=0} = \tilde{f}'(v_k')|_{y=2} = 0.\end{equation} 
\end{cor}
\begin{proof}
    Recall from Proposition \ref{prop:discrete EL} that the Euler--Lagrange equation \eqref{eq:time discrete} holds in duality with $H^1([0,2])$. In particular, for all $r > 0$, we define the truncation 
    \begin{equation*}
        \varphi_r(y) = \min\{\frac{y}{r},1\}-1, \qquad y \in [0,2]. 
    \end{equation*}
    Note that $\partial_y \varphi_r = \frac{1}{r}\mathds{1}_{\{y < r\}}$ in the sense of distributions, whence $\varphi_r \in H^1([0,2])$ for all $r > 0$. Moreover we have $|\varphi_r| \leq 1$. By inserting $\varphi_r$ into the weak formulation of \eqref{eq:time discrete}, we get 
    \begin{equation}\label{eq:return to get vanish at endpoints}
       \int_0^r  \left(\frac{y}{r}-1\right) \Big(\frac{v_k - v_{k-1}}{\tau} + b(y,v_k) \Big) \d y = - \frac{1}{r} \int_0^r \tilde{f}'(v_k') \d y. 
   \end{equation}
  The left-hand side in \eqref{eq:return to get vanish at endpoints} is controlled as follows: 
   \begin{equation*}
      \begin{aligned}
          \bigg|\int_0^r  \underbrace{ \left(\frac{y}{r}-1\right)}_{\in[- 1,0]} \Big(\frac{v_k - v_{k-1}}{\tau} + b(y,v_k) \Big) \d y \bigg| &\leq \int_0^r \Big| \frac{v_k - v_{k-1}}{\tau} + b(y,v_k) \Big| \d y \leq C\sqrt{r}.
      \end{aligned} 
   \end{equation*}
   Returning to \eqref{eq:return to get vanish at endpoints} and letting $r \to 0$ (keeping $k,\tau$ fixed), an application of the Dominated Convergence Theorem and Lebesgue's Differentiation Theorem imply 
   \begin{equation*}
       \tilde{f}'(v_k')|_{y=0} = \lim_{r \to 0} \frac{1}{r}\int_0^r \tilde{f}'(v_k') \d y = 0.
   \end{equation*}
  An analogous procedure yields $\tilde{f}'(v_k')|_{y=2} = 0$ by considering instead $\varphi_r(2-y)$. 
\end{proof}

We conclude this section with the proof of a particular optimality condition which is needed in the sequel. Morally speaking, this says that $(\partial_y v_k)_s$ and $\tilde{f}'(v_k')$ have disjoint supports. 

\begin{lemma}[Optimality condition on the singular part of the derivative]\label{lem:to get rid of bad mixed product}
Let $v_k \in X$ be a minimiser of $\mathscr{E}^k$ provided by Proposition \ref{lem:implicit euler time stepping}. Then, we have 
$\tilde{f}'(v_k')=0$ \, $(\partial_y v_k)_s$-a.e. 
\end{lemma}

\begin{proof}
For clarity of presentation, in this proof, we omit the $k$ subscript in our notation for $v_k$ the minimiser of $\mathscr{E}^k$.

\smallskip 

\noindent 1. \textit{Defining $w$ such that $\partial_y w = (\partial_y v)_s$}. Fix any Lebesgue point $y_0 \in (0,2)$ of $v$, and define 
\begin{equation}\label{eq:def wn approx for sing interplay}
w := v - \int_{y_0}^y v'(s) \d s \in X.
\end{equation}
It follows directly directly from the above that 
\begin{equation}\label{eq:gradient of w only contains singular part}
    \partial_y w = \partial_y v - v' \!\cdot\! \leb = (\partial_y v)_s \geq 0, \qquad w' = 0 \quad \leb\text{-a.e.}, 
\end{equation}
where the inequalities hold in $\mathcal{M}_\loc((0,2))$. In the lines that follow, we verify $w \in X$. 

For $\leb$-a.e.~$y \in (0,2)$, the non-negativity of the measure $(\partial_y v)_s$ in $\mathcal{M}_\loc((0,2))$ and $v' \geq 0$ $\leb$-a.e.~(\textit{cf.}~Proposition \ref{prop:discrete EL}) imply 
\begin{equation*}
  0 \leq \int_{y_0}^y v'(s) \d s \leq \int_{y_0}^y v'(s) \d s + \int_{y_0}^y \, (\partial_y v)_s(s) = \int_{y_0}^y \, \partial_y v(s) = v(y) - v(y_0), 
\end{equation*}
where we used the identification with the precise representative from Remark \ref{rem:precise rep BV} and the formula \eqref{eq:BV FTC}. Hence, returning to \eqref{eq:def wn approx for sing interplay}, we have $0 \leq |w(y)| \leq 2|v(y)| + |v(y_0)|$ $\leb$-a.e.~$y \in (0,2)$. 
As such, by integrating the previous inequality in $y$ we get $\int_0^2 |w(y)|^2 \d y \leq 4|v(y_0)|^2 + 8 \int_0^2 |v(y)|^2 \d y$. Hence $w \in L^2([0,2])$, and it follows that $w \in X$, as required.

\smallskip 

\noindent 2. \textit{Variations along $w$}. Our next objective is to show
\begin{equation}\label{eq:before mollification on w limit in n vanishes}
       \int_0^2 \Big( b(y,v_k) + \frac{v_k-v_{k-1}}{\tau}   \Big) w \d y = 0. 
    \end{equation}
    We do this by computing variations along $w$. Using \eqref{eq:gradient of w only contains singular part}, for all $|\delta| < 1$, we have
\begin{equation}\label{eq:same AC part of deriv gives 0}
    \partial_y(v + \delta w) \geq 0 \, \text{in } \mathcal{M}_\loc((0,2)), \qquad (v+\delta w)' = v'. 
\end{equation}
The minimality of $v$ in $X\ni (v+\delta w)$ implies $\mathscr{E}^k[v] \leq \mathscr{E}^k[v+\delta w]$, from which we deduce 
%\begin{equation}\label{eq:one side of weird EL}
 %  \begin{aligned}
  %     0\leq &\int_0^2 \Big( B(y,v_k +\delta w) - B(y,v_k) \Big) \d y + \frac{1}{2\tau}\int_0^2 ( |v_k\!+\! \delta w \!-\! v_{k-1}|^2 \!-\! |v_k \!-\! v_{k-1}|^2 ) \d y, 
   %\end{aligned}
%\end{equation}
%where we used $\tilde{f}(v_k' + \delta w' ) - \tilde{f}(v_k')=0$ for all $|\delta|<1$ by \eqref{eq:same AC part of deriv gives 0}. By dividing the previous inequality by $\delta > 0$ and letting $\delta \to 0^+$, followed by the same approach with $\delta<0$ and $\delta \to 0^-$ (arguing as per \eqref{eq:just before dct in EL derivation}), we obtain 
\eqref{eq:before mollification on w limit in n vanishes}.

\smallskip 

\noindent 3. \textit{Conclusion}. The strong form \eqref{eq:discrete EL in prop duality L2 with fundamental lem of cov} of the Euler--Lagrange equation allows to test with $w$ without needing to approximate by mollification or cut-off. Multiplying \eqref{eq:discrete EL in prop duality L2 with fundamental lem of cov} with $w$ and integrating, we get 
\begin{equation}\label{eq:wn is the test function}
 0 =  - \int_0^2 w \Big( \frac{v_k-v_{k-1}}{\tau} + b(y,v_k)  \Big) \d y = \int_0^2 w \partial_y(-\tilde{f}'(v_k')) \d y.
\end{equation}
We would like to deduce (by integration by parts) from \eqref{eq:wn is the test function} that we have $\int \tilde f'(v'_k) \partial_y w =0$, which would give the claim. Yet, the difficulty comes from the fact that $\partial_y w$ is not necessarily a finite measure (the boundary term should formally disappear because $\tilde f'(v'_k)\in H^1_0$). Actually, using $\partial_y w\geq 0$ and $\tilde f'(v'_k)\leq 0$, it would be enough to obtain $\int \tilde f'(v'_k) \partial_y w  \geq 0$

The conclusion is obtained using the following facts: $w$ is an $L^2$ function, and $\partial_y w$ is a positive measure; $\tilde f'(v'_k)$ is a non-positive $H^1_0$ function, which can be approximated strongly in $H^1_0$ (and hence uniformly) by a sequence of functions $\varphi_m\in C^1_c$ with $\varphi_m\leq 0$. We then write
$$\int_0^2  \tilde f'(v'_k) \, \partial_y w \geq \lim_{m\to\infty} \int_0^2 \varphi_m \, \partial_y w = - \lim_{m\to\infty} \int_0^2w \partial_y\varphi_m \d y =-\int_0^2 w\partial_y(\tilde f'(v'_k)) \d y =0,$$
where the first inequality above comes from Fatou's lemma, the first equality is an integration by part on a compact subset of $(0,2)$, the next one comes from $w\in L^2$ and the strong $L^2$ convergence of $\partial_y\varphi_m$ to $\partial_y(\tilde f'(v'_k))$ and the last is \eqref{eq:wn is the test function}.
\end{proof}

\section{Solution of the Lagrangian Problem}\label{sec:solution of lagrangian}

In this section, we prove the main well-posedness result for the Lagrangian problem (\textit{cf.}~Theorem \ref{thm:lagrangian}). The main result of this section is the following.

\begin{prop}\label{prop:existence L2 grad flow cts time}
   Let $u_0 \in X$ be such that $ \mathscr{F}[u_0]$ is finite, and $T>0$. Then, there exists $u \in L^\infty(0,T;X) \cap H^1([0,T];L^2([0,2]))$ satisfying $\partial_y u(t,\cdot) \geq 0$ in $\mathcal{M}_\loc((0,2))$ and $0\leq - \tilde{f}'(u') \in L^2(0,T;H^1_0([0,2]))$, and the equation 
   \begin{equation}\label{eq:final weak form continuous time}
  \left\lbrace \begin{aligned} 
     &\partial_t u = \partial_y(\tilde{f}'(u')) - b(y,u), \\ 
     &u|_{t=0} = u_0, 
      \end{aligned} \right. 
   \end{equation}
where the initial data is achieved in the sense $\lim_{t \to 0^+} \Vert u(t,\cdot) - u_0 \Vert_{L^2([0,2])}= 0$,  and with $u'$ as per \eqref{eq:shorthand prime}. Moreover, for $\leb$-a.e.~$t \in [0,T]$, there exists $c_t$ such that 
\begin{equation}\label{eq:essinf in cts time}
    \essinf_{[0,2]} u'(t,\cdot) \geq c_t > 0. 
\end{equation}
\end{prop}

We will prove Proposition \ref{prop:existence L2 grad flow cts time} via a sequence of lemmas. Firstly, we prove uniform estimates for the sequence of piecewise constant time-interpolations $\{u_\tau\}_\tau$. The solution of \eqref{eq:final weak form continuous time} provided by Proposition \ref{prop:existence L2 grad flow cts time} will be obtained as the limit as $\tau \to 0^+$ of the sequence $\{u_\tau\}_\tau$. 

\begin{lemma}[Properties of interpolations]\label{lem:uniform in tau est for interpolations}
    Let  $u_0 \in X$ be such that $\mathscr{F}[u_0]$ is finite, $ 0 < \tau < 1$, $T>0$, and $N := \ceil{T/\tau}$. For $k \in \{1,\dots,N\}$, let the functionals $\mathscr{E}^k$ and their minimisers $v_k$ be as in Proposition \ref{lem:implicit euler time stepping}. For all $t\in [0,T]$, define the piecewise constant interpolation 
    \begin{equation}\label{eq:interp def}
    u_\tau(t,\cdot) := v_k,  \qquad \text{for all } t \in ((k-1)\tau,k\tau], \quad k \in \{0,1,\dots,N\}. 
\end{equation}
Then, $\partial_y u_\tau \geq 0$, $u_\tau(0,\cdot) = u_0$, and for some positive $C=C(\Vert u_0 \Vert_{L^2},\mathscr{E}[u_0],\Vert b \Vert_{L^\infty},T)$ independent of $\tau$, it holds the equicontinuity estimate$:$ 
\begin{equation}\label{eq:equicontinuity in time}
    \Vert u_\tau(t,\cdot) - u_\tau(s,\cdot) \Vert_{L^2([0,2])} \leq  C \big( \sqrt{|t-s|} + \sqrt{\tau} \big) \qquad \text{a.e.~}t,s \in [0,T], 
\end{equation}
$($in particular, the functions $u_\tau(t,\cdot)$ are bounded in $L^2([0,2])$$)$
and the uniform estimates$:$ 
\begin{equation}\label{eq:discrete main L2 est}
\begin{aligned}
     \int_0^2 \tilde{f}(u_\tau') \d y \leq C, \, \Vert u_\tau \Vert_{L^\infty(0,T;L^2([0,2])} \leq C, \, \Vert \tilde{f}'(u_\tau') \Vert_{L^2(0,T;H^1_0([0,2])} \leq C,
\end{aligned} 
\end{equation}
with $u'_\tau$ as per \eqref{eq:shorthand prime}. Moreover, for all compact subsets $[y_0,y_1] \subset (0,2)$, we have
 \begin{equation}\label{eq:unif in tau BV loc est v tau}
     \Vert u_\tau(t,\cdot) \Vert_{TV([y_0,y_1])} \leq C_{y_0,y_1}, 
    \end{equation}
    where $C_{y_0,y_1} = C_{y_0,y_1}(T,y_0,y_1,\Vert u_0 \Vert_{L^2}, \mathscr{E}[u_0],\Vert b \Vert_{L^\infty})$ is independent of $\tau$. Also, we have 
    \begin{equation}\label{eq:bad mixed product vanishes cpct subsets for interp}
\tilde{f}'(u_\tau') =0\quad (\partial_y u_\tau)_s\text{-a.e.,} \quad\leb\text{-a.e.~}t. 
\end{equation}
\end{lemma}

\begin{proof}
Since $\partial_y v_k \geq 0$ in the sense $\mathcal{M}_\loc((0,2))$ for all $k$, direct computation yields: 
\begin{equation}\label{eq:dy vtau has sign}
  0 \leq \partial_y u_\tau(t,\cdot) \quad \text{in } \mathcal{M}_\loc((0,2)),\qquad \leb\text{-a.e.~} t \in [0,T]. 
\end{equation}
By construction, $u_\tau(0,\cdot) = v_0 = u_0 \in X$ and $\mathscr{E}[u_0] \in \mathbb{R}$. Equality \eqref{eq:bad mixed product vanishes cpct subsets for interp} follows   from definition \eqref{eq:interp def}  and Lemma \ref{lem:to get rid of bad mixed product} for $v_k'$. The rest of the proof deals with the uniform-in-$\tau$ estimates.

\smallskip 

\noindent 1. \textit{Estimate on discrete time-derivative}. The minimality of each $v_k$ implies $\mathscr{E}^k[v_k] \leq \mathscr{E}^k[v_{k-1}]$ for all $k \in \{1,\dots,N\}$, \textit{i.e.}, it holds 
\begin{equation}\label{eq:before doing the t est in the unif est}
\frac{1}{2\tau}\Vert v_k - v_{k-1} \Vert^2_{L^2([0,2])} + \mathscr{E}[v_k] \leq \mathscr{E}[v_{k-1}] \qquad \forall k \in \{1,\dots,N\}, 
\end{equation}
from which we deduce $\Vert v_k - v_{k-1} \Vert^2_{L^2([0,2])} \leq 2 \tau ( \mathscr{E}[v_{k-1}] - \mathscr{E}[v_k] )$, and summing over $k$ yields 
\begin{equation}\label{eq:dealing with moment ish term lower bound}
   \begin{aligned}
       \sum_{k=1}^N \Vert v_k - v_{k-1} \Vert^2_{L^2([0,2])} \leq 2 \tau ( \mathscr{E}[v_{0}] - \mathscr{E}[v_N] ) \leq 2\tau \mathscr{E}[u_0] + \tau C\big( 1+ \Vert v_N \Vert_{L^2([0,2])} \big), 
   \end{aligned} 
\end{equation}
where we used the lower bound \eqref{eq:func lower bound for W11}. By the triangle inequality $\Vert v_N \Vert_{L^2([0,2])} \leq \Vert u_0 \Vert_{L^2([0,2])} + \sum_{k=1}^N \Vert v_k - v_{k-1} \Vert_{L^2([0,2])}$, and thus by returning to \eqref{eq:dealing with moment ish term lower bound}, we deduce 
\begin{equation*} 
   \begin{aligned}
       \sum_{k=1}^N\! \Vert v_k \!-\! v_{k-1} \Vert^2_{L^2([0,2])} 
       &\leq \!C\tau\Big( \mathscr{E}[u_0] +\!1 \!+\! \Vert u_0 \Vert_{L^2([0,2])} \Big) \!+\! C \tau \!\!\sum_{k=1}^N \!\Vert v_k \!-\! v_{k-1} \Vert_{L^2([0,2])} \\ 
       &\leq \! C\tau \Big( \mathscr{E}[u_0] \!+\! 1 \!+\!  \Vert u_0 \Vert_{L^2([0,2])} \Big) \!+\! \frac{C^2}{2} N\tau^2 \!+\! \frac{1}{2}\!\sum_{k=1}^N\! \Vert v_k \!-\! v_{k-1} \Vert_{L^2([0,2])}^2, 
   \end{aligned} 
\end{equation*}
where we used Young's inequality in the final sum. In turn, using also that $N\tau \leq T+1$ since $N=\ceil{T/\tau}$, we absorb the final term into the left-hand side, and we find 
\begin{equation}\label{eq:before arzela} 
   \begin{aligned}
       \sum_{k=1}^N \Vert v_k \!-\! v_{k-1} \Vert^2_{L^2([0,2])} 
       \leq C_0\tau, 
   \end{aligned} 
\end{equation}
for some positive $C_0=C_0(T,\Vert u_0 \Vert_{L^2},\Vert b \Vert_{L^\infty},\mathscr{E}[u_0])$ independent of the time-step $\tau$.

\smallskip 

\noindent 2. \textit{Equicontinuity in time}. We rewrite the previous estimate \eqref{eq:before arzela}  in terms of the interpolation \eqref{eq:interp def}. Let $0<t<s<T$ such that $t \in ((n-1)\tau,n\tau]$ and $s \in ((m-1)\tau, m \tau]$; \textit{i.e.}~$n < m$. Then, by the triangle inequality, it holds 
\begin{equation*}
    \Vert u_\tau(t,\cdot) - u_\tau(s,\cdot) \Vert_{L^2([0,2])} \leq \sum_{k=n+1}^{m} \Vert v_{k} - v_{k-1} \Vert_{L^2([0,2])} \leq \bigg(  \sum_{k=n+1}^{m} \Vert v_{k} - v_{k-1} \Vert_{L^2([0,2])}^2 \bigg) ^{\frac{1}{2}} |n-m|^{\frac{1}{2}}, 
\end{equation*}
where we applied the Cauchy--Schwarz inequality; note that the above also holds if $n=m$, since in this case the left-hand side is zero. We note that $|n-m| \leq \frac{|t-s|}{\tau}+1$ by our choice of $t,s$ and thus, by \eqref{eq:before arzela}, we deduce $\Vert u_\tau(t,\cdot) - u_\tau(s,\cdot) \Vert_{L^2([0,2])} \leq  C_0 ( \sqrt{|t-s|} + \sqrt{\tau} )$, where we recall $C_0$ is independent of the time-step $\tau$. We have proved \eqref{eq:equicontinuity in time}.

\smallskip 

\noindent 3. \textit{Uniform $L^\infty_t L^2_y$-bound on $u_\tau$}. Using the triangle inequality as before, 
\begin{equation*}
   \begin{aligned}
       \Vert v_k \Vert_{L^2([0,2])} &\leq \Vert u_0 \Vert_{L^2([0,2])} + \sum_{j=1}^k \Vert v_j - v_{j-1} \Vert_{L^2([0,2])} \\ 
       &\leq \Vert u_0 \Vert_{L^2([0,2])} + \sqrt{N} \bigg(\sum_{j=1}^N \Vert v_j - v_{j-1} \Vert_{L^2([0,2])}^2 \bigg)^{\frac{1}{2}} \leq \Vert u_0 \Vert_{L^2([0,2])} + \sqrt{C_0 N \tau}, 
   \end{aligned} 
\end{equation*}
where we used the Cauchy--Schwarz inequality and  \eqref{eq:before arzela}. Using again $N\tau \leq T+1$, we get that there exists a positive $C_1=C_1(T,\Vert u_0 \Vert_{L^2},\Vert b \Vert_{L^\infty},\mathscr{E}[u_0])$ independent of $\tau$ such that 
\begin{equation}\label{eq:unif L2 est vk}
   \begin{aligned}
       \Vert v_k \Vert_{L^2([0,2])} \leq C_1 \quad \forall k \implies \Vert u_\tau \Vert_{L^\infty(0,T;L^2([0,2]))} \leq C_1. 
   \end{aligned} 
\end{equation}

\smallskip 

\noindent 4. \textit{Estimates on the functional}. By dropping the positive contribution from the $L^2$ bound in \eqref{eq:before doing the t est in the unif est}, we get $\mathscr{E}[v_k] \leq \mathscr{E}[v_{k-1}]$ for all $k$, and thus $\mathscr{E}[v_k] \leq \mathscr{E}[u_0]$. It then follows from the Lipschitz estimate \eqref{eq:B bound} on $B$ and Jensen's inequality that 
\begin{equation*}
\begin{aligned}
    \mathscr{F}[v_k] \!=\! \mathscr{E}[v_k] \!-\! \int_0^2\!\! B(y,v_k) \d y \leq \mathscr{E}[u_0] \!+\! \sqrt{2} \Vert b \Vert_{L^\infty} \Vert v_k \Vert_{L^2([0,2])}  \leq \mathscr{E}[u_0] \!+\! \sqrt{2} \Vert b \Vert_{L^\infty}  C_1 \!=: C_2, 
\end{aligned}
\end{equation*}
using the uniform $L^2$-estimate \eqref{eq:unif L2 est vk}. Hence, $\mathscr{F}[v_k] \leq C_2$, which implies $\int_0^2 \tilde{f}(u_\tau') \d y \leq C_2,$ where $C_2$ is independent of $\tau$. This proves the first part of \eqref{eq:discrete main L2 est}.

\smallskip 

 \noindent 5. \textit{$H^1$-bound on $\tilde{f}'(u_\tau')$}. Directly from estimate \eqref{eq:L2 gradient of bad deriv discrete} of Proposition \ref{prop:discrete EL} and summing in $k$, 
    \begin{equation*}
    \begin{aligned}
      \int_0^T  \Vert \tilde{f}'(u_\tau'(t,\cdot)) \Vert_{H^1([0,2])}^2 \d t \leq C \sum_{k=1}^N \Big( \big\Vert \frac{v_k\!-\!v_{k-1}}{\tau} \big\Vert_{L^2([0,2])}^2  \!+\! \Vert b \Vert_{L^\infty}^2 \Big) \tau \leq C\big(C_0 \!+\! \Vert b \Vert^2_{L^\infty}(1\!+\!T) \big), 
    \end{aligned}
    \end{equation*}
    where the constant $C$ is independent of $\tau$ by  \eqref{eq:before arzela}. We deduce the final part of the estimate \eqref{eq:discrete main L2 est}; moreover,  $\tilde{f}'(u_\tau') \in L^2(0,T;H^1_0([0,2]))$ by Corollary \ref{cor:vanish endpoints}.

\smallskip 

\noindent 6. \textit{$BV_\loc$ bound}. Since $\partial_y u_\tau \geq 0$ from \eqref{eq:dy vtau has sign}, Lemma \ref{lem:blow up ibp} yields 
$$\Vert u_\tau(t,\cdot) \Vert_{TV([y_0,y_1])} \leq \Vert u_\tau \Vert_{L^\infty(0,T;L^2([0,2]))} ( \frac{1}{\sqrt{2-y_1}} + \frac{1}{\sqrt{y_0}} )$$ 
for any compact interval $[y_0,y_1] \subset (0,2)$. Using \eqref{eq:unif L2 est vk} to bound the right-hand side independently of $\tau$, we get \eqref{eq:unif in tau BV loc est v tau}. 
\end{proof}

The next lemma gives a first result concerning the convergence of the sequence $\{u_\tau\}_\tau$ to a limit $u$. It is not sufficient to pass to the limit in the non-linear terms of the Euler--Lagrange equation, but it will subsequently be used to obtain a stronger convergence (\textit{cf.}~Lemma \ref{lem:strong a.e. conv of AC part of derivative}).

\begin{lemma}\label{lem:penultimate taking limit lemma}
    Let $\tau$ and $u_\tau$ be as in Lemma \ref{lem:uniform in tau est for interpolations}. Then, there exists a subsequence of $\{u_\tau\}_\tau$, which we do not relabel, for which we have $u_\tau(t)\to u(t)$ in $X$ for every $t$,  for some $u \in L^\infty(0,T;X) \cap C^{0,\frac{1}{2}}([0,T];L^2([0,2]))$ such that $\partial_y u(t,\cdot) \geq 0$. Furthermore, $\partial_y u_\tau \overset{*}{\rightharpoonup} \partial_y u $ in $L^\infty(0,T;\mathcal{M}_\loc((0,2)))$, and for $\leb$-a.e.~$t$, 
  \begin{equation}\label{eq:weak star conv of measures tau vanish}
      \partial_y u_\tau(t,\cdot) \overset{*}{\rightharpoonup} \partial_y u(t,\cdot) \qquad \text{in } \mathcal{M}_\loc((0,2)). 
  \end{equation} 
  Also, for all $[y_0,y_1] \subseteq [0,2]$, 
    \begin{equation}\label{eq:limit of functional is functional of limit}
       \int_0^T \int_{y_0}^{y_1} \tilde{f}(u') \d y \d t \leq \liminf_{\tau \to 0^+} \int_0^T \int_{y_0}^{y_1} \tilde{f}(u_\tau') \d y \d t, 
    \end{equation}
    where we recall the notations $u_\tau',u'$ of Definition \ref{def:BVloc}. Moreover, $u_\tau$ satisfies 
\begin{equation}\label{eq:weak form for u tau version ii}
  \begin{aligned}
     \bigg| \int_0^{T}\int_0^2 & \bigg(  -\Big(  \frac{\varphi(t+\tau) - \varphi(t)}{\tau} \Big) u_\tau + b(y,u_\tau) \varphi + \tilde{f}'(u_\tau') \partial_y \varphi \bigg) \d y \d t \\ 
      &- \frac{1}{\tau}\int_{0}^\tau \int_0^2 \varphi(t,y) u_0(y) \d y \d t + \frac{1}{\tau}\int_{T-\tau}^T \int_0^2 \varphi(t+\tau) u_\tau(t, y) \d y \d t \bigg| \leq C_\varphi \sqrt{\tau}, 
  \end{aligned}
\end{equation}
for all $\varphi \in C^1([0,T]\!\times\![0,2])$, where the positive constant $C_\varphi$ is independent of $\tau$. 
\end{lemma}

\begin{proof}
By \eqref{eq:equicontinuity in time}, the sequence $\{u_\tau\}_\tau$ satisfies an approximate modulus of continuity in time (with an error of the order of $\sqrt{\tau}$), when valued in $L^2([0,2])$, and hence in $L^1([0,2])$. Moreover, all functions $u_\tau(t,\cdot)$ belong to a bounded subset of $X$, and this subset is compact for the convergence in $X$, \textit{i.e.}~the $L^1$ convergence. A standard variant of the Ascoli--Arzel\`a Theorem (see \cite[Section 2.2]{AGS}), allows to deduce that a subsequence converges uniformly (in $X$) 
towards some limit $u \in C^{0,\frac{1}{2}}([0,T];L^2([0,2]))$ which is valued in the same bounded subset of $X$. Of course the functions $u(t,\cdot)$ also satisfy the condition $\partial_y u \geq 0$ in the sense of $\mathcal{M}_\loc((0,2))$. 

The $BV_\loc$ and $L^\infty_{loc}$ estimate \eqref{eq:unif in tau BV loc est v tau} allows to improve the convergence $u_\tau(t,\cdot)\to u(t,\cdot)$ which is \emph{a priori} stated in $L^1([0,2])$ and have $u_\tau(t,\cdot)\to u(t,\cdot)$ in $L^2_{\loc}((0,2))$ and $
    \partial_y u_\tau(t,\cdot) \overset{*}{\rightharpoonup} \partial_y u(t,\cdot)$ (these conditions hold for every $t$). 
    
    In order to prove \eqref{eq:limit of functional is functional of limit} it is enough to observe that the lower semicontinuity for each $t$, \textit{i.e.~}$  \int_{y_0}^{y_1} \tilde{f}(u') \d y \leq \liminf_{\tau \to 0^+} \int_{y_0}^{y_1} \tilde{f}(u_\tau') \d y $ is a consequence of \eqref{eq:lsc Fb func}. Moreover, the quantity $\int_{y_0}^{y_1} \tilde{f}(u_\tau') \d y$ is bounded from below in terms of $\Vert u_\tau(t,\cdot)\Vert_{L^2([0,2])}$, which is itself uniformly bounded by \eqref{eq:discrete main L2 est}. This allows to obtain a lower bound on the integrand and apply Fatou's lemma to obtain lower semicontinuity after integrating in time.

Multiplying equation \eqref{eq:time discrete} by $\tau$ and summing over $k$, using \eqref{eq:interp def} the definition of $u_\tau$, 
\begin{equation}\label{eq:weak form for u tau version i}
    \int_\tau^T \int_0^2 \bigg( \varphi \Big(\frac{u_\tau(t) - u_\tau(t-\tau)}{\tau} + b(y,u_\tau) \Big) + \tilde{f}'(u_\tau') \partial_y \varphi \bigg) \d y \d t = 0, 
\end{equation}
for all $\varphi \in C^1([0,T]\!\times\![0,2])$. Direct computation yields \eqref{eq:weak form for u tau version ii}, with the final term given by 
\begin{equation*}
    \int_0^\tau \Big( \varphi b(y,u_\tau) + \tilde{f}'(u'_\tau)  \partial_y \varphi \Big) \d y \d t \leq \sqrt{\tau} \Vert \varphi \Vert_{C^1([0,T]\!\times\![0,2])}  (\Vert b \Vert_{L^\infty} + \Vert \tilde{f}'(u'_\tau) \Vert_{L^2([0,T]\!\times\![0,2])}), 
\end{equation*}
whence the constant $C_\varphi$ is given by $\Vert \varphi \Vert_{C^1([0,T]\!\times\![0,2])}  (\Vert b \Vert_{L^\infty} + \sup_\tau \Vert \tilde{f}'(u'_\tau) \Vert_{L^2([0,T]\!\times\![0,2])})$. 
\end{proof}

The next result is fundamental in passing to the limit as $\tau \to 0^+$ in the equation \eqref{eq:weak form for u tau version ii}. 

\begin{lemma}[Absolutely continuous parts converge a.e.]\label{lem:strong a.e. conv of AC part of derivative}
    There exists a subsequence of $\{u_\tau\}_\tau$ $($which we do not relabel$)$ such that $u'_\tau \to u'$ $\leb$-a.e. 
\end{lemma}

The proof of this result relies on the convexity of $\tilde{f}$ and an argument inspired by the Minty--Browder monotonicity method for compactness (\textit{cf.~e.g.~}\cite[\S 2 and \S 5]{EvansWeak}). We note in passing that this situation is specific to the gradient flow structure of the problem; \textit{e.g.}~for any approximation of the identity, the absolutely continuous part of the gradients does not converge $\leb$-a.e.~(since the gradient concentrates as $\delta_0'$ the derivative of the Dirac). 

\begin{proof}
 The convexity estimate \eqref{eq:convexity for tilde f second gives L2 bound} applied to $s_1 = u'$ and $s_2 = u_\tau'$ yields 
\begin{equation*}
        \tilde{f}(u') - \tilde{f}(u_\tau') - \tilde{f}'(u_\tau')(u' - u_\tau')  \geq \omega(u' , u_\tau') \qquad \leb\text{-a.e.} 
    \end{equation*}
 By integrating only in the $y$ variable on a compact interval $[y_0,y_1] \subset (0,2)$, we get 
\begin{equation}\label{eq:first bound with convexity final a.e. conv}
     \begin{aligned}  \int_{y_0}^{y_1}\! \big( \tilde{f}(u')  \!-\! \tilde{f}(u_\tau')  \big) \d y + \int_{y_0}^{y_1}\!  -\tilde{f}'(u_\tau') (u' \!-\! u_\tau')  \d y \geq \int_{y_0}^{y_1} \omega(u',u'_\tau) \d y \qquad \leb\text{-a.e.~}t. 
     \end{aligned}
    \end{equation}
  To avoid heavy notation, we omit the explicit dependence on $t$ in what follows. We estimate the second integral in \eqref{eq:first bound with convexity final a.e. conv}. Using the identification with the precise representative from Remark \ref{rem:precise rep BV} and the formula \eqref{eq:BV FTC}, we have 
\begin{equation}\label{eq:the return to get bad mixed prod convergence}
   \begin{aligned}
       \int_{y_0}^{y_1}\!\!\!\!\! -\tilde{f}'(u_\tau') (u' \!-\! u_\tau') \d y =& \! \int_{y_0}^{y_1} \!\!\!\!\!-\tilde{f}'(u_\tau') \, \partial_y(u \!-\! u_\tau\!)\! -\!\! \int_{y_0}^{y_1} \!\!\!\!\! -\tilde{f}'(u_\tau') \,  (\partial_y u)_s + \!\!\int_{y_0}^{y_1}\!\!\!\!\! -\tilde{f}'(u_\tau') \, (\partial_y u_\tau\!)_s \\ 
       =& \, \big[ -\tilde{f}'(u_\tau') (u\!-\!u_\tau) \big]_{y_0}^{y_1} + \int_{y_0}^{y_1} (u\!-\!u_\tau) \partial_y ( \tilde{f}'(u'_\tau) ) \d y \\ 
       &- \underbrace{\int_{y_0}^{y_1}  -\tilde{f}'(u_\tau') \,  (\partial_y u)_s}_{\geq 0} + \underbrace{\int_{y_0}^{y_1}  -\tilde{f}'(u_\tau') \, (\partial_y u_\tau)_s}_{=0}, 
   \end{aligned}
\end{equation}
where we integrated by parts to obtain the second line, used $(\partial_y u)_s \geq 0$ in $\mathcal{M}_\loc((0,2))$ by Lemma \ref{lem:penultimate taking limit lemma}, and used Lemma \ref{lem:to get rid of bad mixed product} to make the final term vanish. We deduce 
\begin{equation*}
   \begin{aligned}
      \int_{y_0}^{y_1}\!\!\! -\! \tilde{f}'(u_\tau') (u' \!-\! u_\tau') \d y  \leq \big[\! -\!\tilde{f}'(u_\tau') (u\!-\!u_\tau) \big]_{y_0}^{y_1} \!+\! \Vert u \!-\! u_\tau \Vert_{L^2([y_0,y_1])} \Vert \tilde{f}'(u_\tau') \Vert_{H^1_0([0,2])},  
   \end{aligned}
\end{equation*}
and we estimate the boundary terms on the right-hand side as follows: using the $BV_\loc$ estimate of Lemma \ref{lem:blow up ibp}, $u_\tau \in L^\infty(0,T;X)$, and $-\tilde{f}'(u_\tau') \geq 0$, we have the one-sided esimate 
\begin{equation}\label{eq:one sided important i}
    \begin{aligned}
        \big[\! -\!\tilde{f}'(u_\tau') (u\!-\!u_\tau) \big]_{y_0}^{y_1} \!= & -\!\!\tilde{f}'(u_\tau')u \big|_{y_1} \!+\! \tilde{f}'(u_\tau')  u\big|_{y_0} \!+\! \tilde{f}'(u_\tau')  u_\tau\big|_{y_1}   \!-\!\tilde{f}'(u_\tau')u_\tau\big|_{y_0} \\ 
        \leq & -\tilde{f}'(u_\tau'(y_1)) \frac{\Vert u \Vert_{L^2([y_1,2])}}{\sqrt{2-y_1}} - \tilde{f}'(u_\tau'(y_0)) \frac{\Vert u \Vert_{L^2([0,y_0])}}{\sqrt{y_0}}  \\ 
        & -\tilde{f}'(u'_\tau(y_1)) \frac{\Vert u_\tau \Vert_{L^2([0,2])}}{\sqrt{y_1}} - \tilde{f}'(u_\tau'(y_0)) \frac{\Vert u_\tau \Vert_{L^2([0,2])}}{\sqrt{2-y_0}}. 
    \end{aligned}
\end{equation}
As $\tilde{f}'(u_\tau') \in H^1_0([0,2])$, we compute $- \tilde{f}'(u_\tau'(y_1)) = \int_{y_1}^2 \partial_y (\tilde{f}'(u_\tau')) \d y \leq \sqrt{2-y_1} \Vert \tilde{f}'(u_\tau') \Vert_{H^1([0,2])}$ and $- \tilde{f}'(u_\tau'(y_0)) = -\int_{0}^{y_0} \partial_y (\tilde{f}'(u_\tau')) \d y \leq \sqrt{y_0} \Vert \tilde{f}'(u_\tau') \Vert_{H^1([0,2])}$. By returning to \eqref{eq:one sided important i}, we get 
\begin{equation}\label{eq:one sided important ii}
    \begin{aligned}
        \big[\! -\!\tilde{f}'(u_\tau') (u\!-\!u_\tau) \big]_{y_0}^{y_1} 
        \leq & \, \Vert \tilde{f}'(u_\tau') \Vert_{H^1_0([0,2])} \Big( \Vert u \Vert_{L^2([y_1,2])} +  \Vert u \Vert_{L^2([0,y_0])} \Big)   \\ 
        & + \Vert \tilde{f}'(u_\tau') \Vert_{H^1_0([0,2])} \Vert u_\tau \Vert_{L^2([0,2])} \Big( \frac{\sqrt{2-y_1}}{\sqrt{y_1}} + \frac{\sqrt{y_0}}{\sqrt{2-y_0}} \Big), 
    \end{aligned}
\end{equation}
By integrating \eqref{eq:first bound with convexity final a.e. conv} in time and using H\"older's inequality, we get for all $[y_0,y_1]\subset (0,2)$ 
\begin{equation*}
     \begin{aligned}  
     \int_0^T \!\!\int_{y_0}^{y_1}\! \omega(u',u_\tau') \d y \d t \leq &\int_0^T \!\!\int_{y_0}^{y_1}\!  \big( \tilde{f}(u') \!-\! \tilde{f}(u_\tau')  \big) \d y \d t \\ 
    &+ \Vert \tilde{f}'(u_\tau') \Vert_{L^2(0,T;H^1_0([0,2]))} \Big( \Vert u \Vert_{L^2(0,T;L^2([y_1,2]))} +  \Vert u \Vert_{L^2(0,T;L^2([0,y_0]))} \Big)   \\ 
        & + \Vert \tilde{f}'(u_\tau') \Vert_{L^2(0,T;H^1_0([0,2]))} \Vert u_\tau \Vert_{L^2(0,T;L^2([0,2]))} \Big( \frac{\sqrt{2-y_1}}{\sqrt{y_1}} + \frac{\sqrt{y_0}}{\sqrt{2-y_0}} \Big) \\ 
     &+ \Vert u \!-\! u_\tau \Vert_{L^2(0,T;L^2([y_0,y_1]))} \Vert \tilde{f}'(u_\tau') \Vert_{L^2(0,T;H^1_0([0,2]))}, 
     \end{aligned}
    \end{equation*}
    By Lemma \ref{lem:uniform in tau est for interpolations}, $\Vert \tilde{f}'(u_\tau') \Vert_{L^2(0,T;H^1_0([0,2]))},\Vert u_\tau \Vert_{L^2(0,T;L^2([0,2]))} \leq C$ independent of $\tau$, whence 
    \begin{equation*}
     \begin{aligned}  \int_0^T \!\!\int_{y_0}^{y_1}\! \omega(u',u_\tau') \d y \d t \!\leq\!  & \, \int_0^T \!\!\int_{y_0}^{y_1}\! \big( \tilde{f}(u')  \!-\! \tilde{f}(u_\tau')  \big) \d y \d t  + C \Vert u \!-\! u_\tau \Vert_{L^2(0,T;L^2([y_0,y_1]))}  \\ 
     &+ C \bigg( \Vert u \Vert_{L^2(0,T;L^2([y_1,2]))} \!+\!  \Vert u \Vert_{L^2(0,T;L^2([0,y_0]))} \!+\! \frac{\sqrt{2\!-\!y_1}}{\sqrt{y_1}} \!+\! \frac{\sqrt{y_0}}{\sqrt{2\!-\!y_0}} \bigg). 
     \end{aligned}
    \end{equation*}
    Since $u_\tau \to u$ in $L^\infty(0,T;L^2([y_0,y_1]))$ for all $[y_0,y_1] \subset (0,2)$, using the lower semicontinuity result \eqref{eq:limit of functional is functional of limit}, we let $\tau \to 0$ (for the subsequence that achieves the $\liminf$ in \eqref{eq:limit of functional is functional of limit}) and get 
    \begin{equation*}
     \begin{aligned}  
     \lim_{\tau \to 0}  \int_0^T \!\!\int_{y_0}^{y_1}\! \omega(u',u_\tau') \d y \d t \leq C \bigg( \Vert u \Vert_{L^2(0,T;L^2([y_1,2]))} \!+\!  \Vert u \Vert_{L^2(0,T;L^2([0,y_0]))} \!+\! \frac{\sqrt{2\!-\!y_1}}{\sqrt{y_1}} \!+\! \frac{\sqrt{y_0}}{\sqrt{2\!-\!y_0}} \bigg).
     \end{aligned}
    \end{equation*}
    It is clear that the left-hand side in this inequality increases if we replace $[y_0,y_1]$ with a larger interval $[y_0',y_1']$, so that we have
\begin{equation*}
     \begin{aligned}  
     \lim_{\tau \to 0}  \int_0^T \!\!\int_{y_0}^{y_1}\! \omega(u',u_\tau') \d y \d t \leq C \bigg( \Vert u \Vert_{L^2(0,T;L^2([y_1',2]))} \!+\!  \Vert u \Vert_{L^2(0,T;L^2([0,y_0']))} \!+\! \frac{\sqrt{2\!-\!y_1'}}{\sqrt{y_1'}} \!+\! \frac{\sqrt{y_0'}}{\sqrt{2\!-\!y_0'}} \bigg).
     \end{aligned}
    \end{equation*}
    We then consider $y_0'\to 0$ and $y_1'\to 2$ so that the right-hand side tends to $0$ and obtain
     \begin{equation}\label{eq:almost there boy!}
    \lim_{\tau \to 0}  \int_0^T \!\!\int_{y_0}^{y_1}\! \omega(u',u_\tau') \d y \d t =0
    \end{equation}
    for any $0<y_0<y_1<2$. From \eqref{eq:almost there boy!}, we have (for a subsequence which we do not relabel) 
\begin{equation}\label{eq:a.e. weird conv}
   \lim_{\tau \to 0} \omega(u',u_\tau') = 0 \qquad \leb\text{-a.e.}
\end{equation}
The conditions on $\omega$, which only vanishes on the diagonal together with conditions \eqref{eq:bad limits for omega}, imply 
$u_\tau' \to u'$ $\leb$-a.e.
\end{proof}

We are ready to give the proof of the main result of this section. 

\begin{proof}[Proof of Proposition \ref{prop:existence L2 grad flow cts time}]

The proof is divided into four steps. We fix a test function $\varphi \in C^1_c((0,T)\!\times\!(0,2))$. We assume $\supp \varphi(t,\cdot) \subseteq [y_0,y_1]$ for all $t \in [0,T]$, for some $[y_0,y_1] \subset (0,2)$. 

\smallskip 

\noindent 1. \textit{Passing to the limit in $\tilde{f}'(u_\tau')$}. The continuity of $\tilde{f}'$ and the $\leb$-a.e.~convergence $u_\tau' \to u'$ of Lemma \ref{lem:strong a.e. conv of AC part of derivative} implies $\tilde{f}'(u_\tau') \to \tilde{f}'(u')$ $\leb$-a.e. Additionally, since $\{\tilde{f}'(u_\tau')\}_\tau$ is a bounded sequence in $L^2(0,T;H^1_0([0,2]))$ from the estimates of Lemma \ref{lem:uniform in tau est for interpolations}, we deduce that 
\begin{equation}\label{eq:weak conv of f' u' tau}
    \tilde{f}'(u_\tau') \rightharpoonup \tilde{f}'(u') \quad \text{weakly in } L^2(0,T;H^1_0([0,2])), 
\end{equation}
whence $\tilde{f}'(u') \in L^2(0,T;H^1_0([0,2]))$ and $\lim_{\tau } \int_0^T \int_0^2 \tilde{f}'(u_\tau') \partial_y \varphi \d y \d t = \int_0^T \int_0^2 \tilde{f}'(u') \partial_y \varphi \d y \d t$. 

\smallskip 

\noindent 2. \textit{Derivative lower bound}. The previous step implies $\tilde{f}'(u') \in L^2(0,T;H^1_0([0,2]))$, whence 
\begin{equation}\label{eq:to get that essinf cts time 0}
\int_0^T \Vert \tilde{f}'(u') \Vert^2_{H^1_0([0,2])} \d t < \infty \implies \Vert \tilde{f}'(u') \Vert_{H^1_0([0,2])} < \infty \quad \leb\text{-a.e.~}t.
\end{equation}
By Morrey's embedding, we deduce from \eqref{eq:to get that essinf cts time 0} that $\Vert \tilde{f}'(u') \Vert_{L^\infty([0,2])} < \infty$ for $\leb$-a.e.~$t$, and 
\begin{equation}\label{eq:to get the essinf in cts time}
  0 \leq -  \tilde{f}'(u') \leq \Vert \tilde{f}'(u'(t,\cdot)) \Vert_{L^\infty([0,2])} \qquad \leb\text{-a.e.}, 
\end{equation}
where we used $\tilde{f}' < 0$. Since $\tilde{f}'$ is strictly increasing by the strict convexity of $\tilde{f}$, it has a well-defined inverse, and we set $$c_t := (-\tilde{f}')^{-1}\big(\Vert \tilde{f}'(u'(t,\cdot)) \Vert_{L^\infty([0,2])} \big) > 0;$$ 
this quantity is strictly positive since $\lim_{s \to 0}\tilde{f}'(s) = - \infty$ by \eqref{eq:f tilde prime blows up at zero assumption-1} and $\tilde{f}$ is strictly decreasing. Estimate \eqref{eq:essinf in cts time} now follows from applying the inverse $(-\tilde{f}')^{-1}$ to the whole inequality \eqref{eq:to get the essinf in cts time}.

\smallskip 

\noindent 3. \textit{Passing to the limit in $b(y,u_\tau)$}. Recall that the test function $\varphi$ has $\supp \varphi(t,\cdot) \subseteq [y_0,y_1] \subset (0,2)$ for all $t $. Using the explicit form \eqref{eq:b def} for $b$ and writing $A^1 = A, A^2=A^c$, we get 
\begin{equation}\label{eq:from b to explicit before Lusin}
   \begin{aligned}
       \int_0^2 b(y,u_\tau) \varphi \d y = \int_{[y_0,y_1]\cap A} \!\!\!\!\!\!\!\!  \!\!\! \overbrace{\partial_x V}^{=:b^1}(u_\tau) \varphi \d y \!+\! \int_{[y_0,y_1]\cap A^c} \!\!\!\! \!\!\!\!\!\!\!\overbrace{\partial_x W}^{=:b^2}(u_\tau) \varphi \d y = \sum_{j=1}^2 \int_{[y_0,y_1] \cap A^j} \!\!\! \!\!\!\!\!\!\!\! b^j(u_\tau) \varphi \d y , 
   \end{aligned}
\end{equation}
and the same decomposition holds for $\int_0^2 b(y,u)\varphi \d y$ with $u_\tau$ replaced by $u$ on the right. We shall approximate $b^1=\partial_x V$ and $b^2 = \partial_x W$, which are merely $L^\infty$, by continuous functions using Lusin's Theorem; we then pass to the limit in $\tau$ without requiring continuity of $b_1,b_2$.

To this end, notice that the uniform bounds of Lemma \ref{lem:uniform in tau est for interpolations} and Lemma  \ref{lem:blow up ibp} imply 
\begin{equation*}
  \Vert u \Vert_{L^\infty([0,T]\times[y_0,y_1])} +   \Vert u_\tau \Vert_{L^\infty([0,T]\times[y_0,y_1])} \leq C\Big( \frac{1}{\sqrt{y_0}} + \frac{1}{\sqrt{2-y_1}} \Big) =: C_{y_0,y_1}, 
\end{equation*}
where the constant $C_{y_0,y_1}$ is independent of $\tau$. We set $I = [-C_{y_0,y_1},C_{y_0,y_1}]$, and remark that it suffices to consider the restrictions of $b^1,b^2$ to the bounded subset $I$ instead of all of $\mathbb{R}$. By Lusin's Theorem, for all $\varepsilon>0$ there exists compact subsets $K_{j,\varepsilon} \subseteq I$ such that $\leb(I\setminus K_{j,\varepsilon}) <\varepsilon$ and $b^j$ coincides on ${K_{j,\varepsilon}}$ with a continuous function (denoted $b^j_\varepsilon$) bounded by the same constants which bound $b^j$ ($j=1,2$).  Define the error terms $R^j_\varepsilon := b^j - b^j_\varepsilon$, and note that $R^j_\varepsilon = 0$ in $K_{j,\varepsilon}$ and $\Vert R^j_\varepsilon \Vert_{L^\infty} \leq 2\Vert b^j \Vert_{L^\infty} < \infty$ since $\partial_x V,\partial_x W \in L^\infty$. 
Then, for $j=1,2$ 
\begin{equation}\label{eq:lusin for b i}
    \int_{[y_0,y_1] \cap A^j}  b^j(u_\tau) \varphi \d y   = \int_{[y_0,y_1]\cap A^j} b^j_\varepsilon(u_\tau) \varphi \d y   + \int_{B^j_{\varepsilon,\tau}} R^j_\varepsilon(u_\tau) \varphi \d y  \qquad \leb\text{-a.e.~}t, 
\end{equation}
where $B_{\varepsilon,\tau}^j = \{ y \in [y_0,y_1] \cap A^j : u_\tau(t,y) \in K_{j,\varepsilon}^c \}$, and we emphasise that we identify $u_\tau, u$ with their precise representatives of Remark \ref{rem:precise rep BV}. We estimate the final term as 
\begin{equation}\label{eq:how to est final term in Bj}
   \bigg| \int_{B^j_{\varepsilon,\tau}} R^j_\varepsilon(u_\tau) \varphi \d y   \bigg| \leq 2 \Vert b^j \Vert_{L^\infty} \Vert \varphi \Vert_{L^\infty} \leb(B^j_{\varepsilon,\tau}). 
\end{equation}
Let $B^j_{\varepsilon} = \{ y \in [y_0,y_1] \cap A^j : u(t,y) \in K_\varepsilon^c \}$ and $b^j(u)\!-\!b^j_\varepsilon(u) = R^j_\varepsilon(u)$. Returning to \eqref{eq:lusin for b i}, 
\begin{equation}\label{eq:the bj limit in tau}
    \begin{aligned}
        \bigg|   \int_{[y_0,y_1]\cap A^j}  \big( b^j(u_\tau) - b^j(u) \big) \varphi \d y   \bigg|
       \leq  & \,  \bigg| \int_{[y_0,y_1]\cap A^j}    \big( b^j_\varepsilon(u_\tau) - b^j_\varepsilon(u) \big) \varphi \d y \bigg|   \\ 
       &+ 2 \Vert b^j \Vert_{L^\infty}   \Vert \varphi \Vert_{L^\infty}  \Big(  \leb(B^j_{\varepsilon}) +   \leb(B^j_{\varepsilon,\tau}) \Big), 
    \end{aligned}
\end{equation}
where, as per \eqref{eq:how to est final term in Bj}, we also used the estimate $|  \int_{B^j_\varepsilon}   R^j_\varepsilon(u) \varphi \d y | \leq 2 \Vert b \Vert_{L^\infty}  \Vert \varphi \Vert_{L^\infty} \leb(B^j_{\varepsilon})$. 

It remains to estimate $\leb(B^j_\varepsilon)$ and $\leb(B^j_{\varepsilon,\tau})$. We rewrite $B^j_{\varepsilon,\tau}$ as a preimage set, which yields 
\begin{equation}\label{eq:estimating the measure of Bj bad set}
    \leb(B^j_{\varepsilon,\tau}) = \leb \big( [y_0,y_1] \cap A^j \cap u_\tau(t,\cdot)^{-1}(K^c_{j,\varepsilon}) \big); 
\end{equation}
analogously $\leb(B^j_\varepsilon) = \leb([y_0,y_1] \cap A^j \cap u(t,\cdot)^{-1}(K^c_{j,\varepsilon}))$. We estimate the right-hand side: recall that for every Borel measurable set $E\subset [0,2]$ and for strictly increasing $v \in BV_{\loc}$, 
\begin{equation*}
    \int_{E} v' \d y \leq \leb(v(E)), 
\end{equation*}
where $v(E)$ is the image set corresponding to the precise representative of Remark \ref{rem:precise rep BV} and $v'$ the absolutely continuous part of $\partial_y v$. To obtain this inequality it is enough to use the measure-valued version of the coarea formula for $BV$ functions presented in section 5.5 of 
\cite{EvansGariepy}.

We then obtain
\begin{equation*}
   \begin{aligned}
       \leb(B^j_{\varepsilon,\tau}) \cdot\!\! \underbrace{\essinf_{[0,2]} u'_\tau }_{>0 \text{ by Lemma \ref{lem:lower bound competitor}}} \leq \int_{B^j_{\varepsilon,\tau}} u_\tau' \d y      
       \leq \leb(K^c_{j,\varepsilon}) < \varepsilon.
   \end{aligned} 
\end{equation*}
Arguing analogously for $\leb(B^j_\varepsilon)$, we obtain the estimates 
\begin{equation}\label{eq:estimating the measure of Bj bad set good}
    \leb(B^j_{\varepsilon,\tau}) \leq \frac{\varepsilon}{\essinf_{[0,2]} u'_\tau(t,\cdot)}, \qquad  \leb(B^j_{\varepsilon}) \leq \frac{\varepsilon}{\essinf_{[0,2]} u'(t,\cdot)}, 
\end{equation}
for $\leb$-a.e.~$t$, and by returning to \eqref{eq:the bj limit in tau}, we get 
\begin{equation}\label{eq:hopefully the last step of b proof}
    \begin{aligned}
        \bigg|   \int_{[y_0,y_1]\cap A^j} \!\! \big( b^j(u_\tau) \!-\! b^j(u) \big) \varphi \d y   \bigg|
       \leq  & \,  \bigg| \int_{[y_0,y_1]\cap A^j}  \!\!  \big( b^j_\varepsilon(u_\tau) \!-\! b^j_\varepsilon(u) \big) \varphi \d y \bigg|   \\ 
       &+ 2 \varepsilon   \Vert b^j \Vert_{L^\infty} \Vert \varphi \Vert_{L^\infty}  \Big(  \frac{1}{\essinf  u'_\tau(t,\cdot)}\! +\! \frac{1}{\essinf  u'(t,\cdot)} \Big). 
    \end{aligned}
\end{equation}

Next, we split our analysis into two parts: for times $t$ which are such that $\essinf_{[0,2]} u'_\tau(t,\cdot)$ is small, and times for which it is large. To this end, for each $\tau$, define the set of times 
\begin{equation*}
    \mathcal{T}_{\varepsilon,\tau} := \{t \in [0,T]: \, \essinf_{[0,2]} u'_\tau(t,\cdot) \geq \sqrt{\varepsilon} \} \cap \{t \in [0,T]: \, \essinf_{[0,2]} u'(t,\cdot) \geq \sqrt{\varepsilon} \}. 
\end{equation*}
Note that the sign and monotonicity of $\tilde{f}'$ implies 
\begin{equation*}
\begin{aligned}
 \leb(\mathcal{T}_{\varepsilon,\tau}^c) &\leq \leb(\{t : \essinf  u'_\tau(t,\cdot) < \sqrt{\varepsilon} \}) + \leb(\{t :  \essinf  u'(t,\cdot) < \sqrt{\varepsilon} \}) \\ 
 &=\leb(\{ t:  \Vert \tilde{f}'(u_\tau'(t,\cdot)) \Vert_{L^\infty}^2 >  |\tilde{f}'(\sqrt{\varepsilon})|^2 \}) + \leb(\{ t: \Vert \tilde{f}'(u'(t,\cdot)) \Vert_{L^\infty}^2 >  |\tilde{f}'(\sqrt{\varepsilon})|^2 \}) \\ 
 &\leq \frac{C}{|\tilde{f}'(\sqrt{\varepsilon})|^2}\Vert \tilde{f}'(u_\tau') \Vert_{L^2(0,T;H^1([0,2]))}^2 + \frac{C}{|\tilde{f}'(\sqrt{\varepsilon})|^2}\Vert \tilde{f}'(u') \Vert_{L^2(0,T;H^1([0,2]))}^2 \leq \frac{C}{|\tilde{f}'(\sqrt{\varepsilon})|^2}, 
\end{aligned}
\end{equation*}
where we used Markov's inequality and Morrey's embedding to obtain the penultimate inequality, and the uniform bound of Lemma \ref{lem:uniform in tau est for interpolations} in the final one. By returning to \eqref{eq:hopefully the last step of b proof}, 
\begin{equation*}\label{eq:hopefully the last step of b proof i}
    \begin{aligned}
      \int_{\mathcal{T}_{\varepsilon,\tau}} \! \bigg|  \!\! \int_{[y_0,y_1]\cap A^j} \!\!\!\!\! \! \!\! \! \big( b^j(u_\tau) \!-\! b^j(u) \big) \varphi \d y  \bigg| \d t
       \!\leq\!  \int_0^T \! \bigg| \!\! \int_{[y_0,y_1]\cap A^j} \! \!\!\!\!\!\!\! \! \big( b^j_\varepsilon(u_\tau) \!-\! b^j_\varepsilon(u) \big) \varphi \d y \bigg| \d t  \!+\! 4\sqrt{\varepsilon} T  \Vert b^j \Vert_{L^\infty}\! \Vert \varphi \Vert_{L^\infty}, 
    \end{aligned}
\end{equation*}
while, on the other hand, 
\begin{equation*}
    \begin{aligned}
        \int_{\mathcal{T}^c_{\varepsilon,\tau}} \bigg|   \int_{[y_0,y_1]\cap A^j} \!\! \big( b^j(u_\tau) \!-\! b^j(u) \big) \varphi \d y   \bigg| \d t \leq 2 \Vert b^j\Vert_{L^\infty} \leb(\mathcal{T}_{\varepsilon,\tau}^c) \leq \frac{C \Vert b^j \Vert_{L^\infty}}{|\tilde{f}'(\sqrt{\varepsilon})|^2}. 
    \end{aligned}
\end{equation*}
Putting the previous two estimates together and letting $\tau \to 0^+$, using the continuity of $b^j_\varepsilon$, the convergence $u_\tau \to u$ $\leb$-a.e.~in $[0,T]\!\times\![0,2]$, and the Dominated Convergence Theorem, 
\begin{equation*}
    \begin{aligned}
        \limsup_{\tau \to 0^+} \int_0^T \bigg|   \int_{[y_0,y_1]\cap A^j} \!\! \big( b^j(u_\tau) \!-\! b^j(u) \big) \varphi \d y   \bigg| \d t \leq   4\sqrt{\varepsilon} T  \Vert b^j \Vert_{L^\infty}\! \Vert \varphi \Vert_{L^\infty} + \frac{C \Vert b^j \Vert_{L^\infty}}{|\tilde{f}'(\sqrt{\varepsilon})|^2}. 
    \end{aligned}
\end{equation*}
The left side is independent of $\varepsilon$. By letting $\varepsilon \to 0$ and using $\tilde{f}'(\sqrt{\varepsilon}) \to -\infty$ by \eqref{eq:f tilde prime blows up at zero assumption-1}, we get 
\begin{equation*}
    \begin{aligned}
        \lim_{\tau \to 0^+} \bigg|   \int_0^T \int_{[y_0,y_1]\cap A^j} \!\! \big( b^j(u_\tau) \!-\! b^j(u) \big) \varphi \d y  \d t \bigg| = 0, 
    \end{aligned}
\end{equation*}
and, by returning to \eqref{eq:from b to explicit before Lusin} and using the triangle inequality, it follows that 
\begin{equation*}
    \lim_{\tau \to 0^+} \int_0^T  \int_0^2 b(y,u_\tau) \varphi \d y \d t = \int_0^T  \int_0^2 b(y,u) \varphi \d y \d t.
\end{equation*}

The convergence $u_\tau \to u$ in $L^\infty(0,T;L^2([y_0,y_1]))$ for all $[y_0,y_1]\subset (0,2)$ from Lemma \ref{lem:penultimate taking limit lemma} is sufficient to pass to the limit in all the other terms of the weak formulation \eqref{eq:weak form for u tau version ii}, and also to preserve the continuity valued in $L^2$ of the limit curve. %The regularity of the limit is such that the equation holds in the strong $\leb$-a.e.~sense of \eqref{eq:final weak form continuous time}. 

   \smallskip

   \noindent 4. \textit{Convergence to the initial data}. Setting $s=0$ in \eqref{eq:equicontinuity in time}, we get $\Vert u_\tau(t,\cdot) - u_0 \Vert_{L^2([0,2])} \leq C_0(\sqrt{t} + \sqrt{\tau})$.%, which yields $\liminf_{\tau\to 0} \Vert u_\tau(t,\cdot) - u_0 \Vert_{L^2([0,2])} \leq C_0 \sqrt{t}$. 
   Meanwhile, the convergence $u_\tau(t) \to u(t)$ in $X$ implies %$u_\tau(t,\cdot) \to u(t,\cdot)$ $\leb$-a.e.~for a.e.~$t$, and by Fatou's Lemma 
%    \begin{equation*}
 %    \begin{aligned}  \Vert u(t,\cdot) - u_0 \Vert_{L^2([0,2])}^2 = \int_0^2 |u(t,\cdot) - u_0|^2 \d y &\leq \liminf_{\tau \to 0} \int_0^2 |u_\tau(t,\cdot) - u_0|^2 \d y \\ 
  %   &= \liminf_{\tau \to 0}\Vert u_\tau(t,\cdot) - u_0 \Vert^2_{L^2([0,2])} \leq C_0^2 t. 
   %  \end{aligned}
   %\end{equation*}
   %Hence, 
   $ \Vert u(t,\cdot) - u_0 \Vert_{L^2([0,2])} \leq C_0\sqrt{t} \to 0$ as $t \to 0^+$. The proof is complete. 
\end{proof}

Finally, we record the continuous-time version of the optimality condition of Lemma \ref{lem:to get rid of bad mixed product}. 

\begin{lemma}[Condition on the singular part of the derivative]\label{lem:cts time optimality condition bad product}
   The function 
  $\tilde{f}'(u')$ vanishes $ (\partial_y u)_s$-a.e. for  $\leb\text{-a.e.~}t \in [0,T].$ 
\end{lemma}
\begin{proof}
For clarity of presentation, we select a discrete subsequence $\{u_{\tau_n}\}_n$ of the sequence $\{u_\tau\}_\tau$ with which we pass to the limit in Proposition \ref{prop:existence L2 grad flow cts time}. We write $u_n$ in place of $u_{\tau_n}$.

 \smallskip 
 
 \noindent 1. \textit{Boundedness of subsequences on time-slices}. Recall from \eqref{eq:weak conv of f' u' tau} and the argument that precedes it that (for a subsequence which we do not relabel here) we have $\tilde{f}'(u_n') \to \tilde{f}'(u')$ $\leb$-a.e.~and weakly in $L^2(0,T;H^1_0([0,2]))$. The purpose of this step is to show, for $\leb$-a.e.~$t \in [0,T]$, that there exists a subsequence $\{\tilde{f}'(u_{\sigma_t(n)}'(t,\cdot))\}_n$ depending on $t$ and $C_t>0$ such that 
 \begin{equation}\label{eq:weak convergence time slice}
  \sup_n \Vert \tilde{f}'(u_{\sigma_t(n)}'(t,\cdot)) \Vert_{H^1_0([0,2])} \leq C_t < \infty. 
 \end{equation}

The proof of \eqref{eq:weak convergence time slice} follows directly from Fatou's Lemma. Indeed, we have
\begin{equation*}
    \begin{aligned}
        \int_0^T \liminf_{n\to\infty} \Vert \tilde{f}'(u_n'(t,\cdot)) \Vert_{H^1_0([0,2])}^2 \d t &\leq \liminf_{n\to\infty} \int_0^T \Vert \tilde{f}'(u_n'(t,\cdot)) \Vert_{H^1_0([0,2])}^2 \d t \leq C, 
    \end{aligned}
\end{equation*}
where we used \eqref{eq:discrete main L2 est}. Hence $\liminf_{n\to\infty} \Vert \tilde{f}'(u_n'(t,\cdot)) \Vert_{H^1_0([0,2])}^2$ is finite for $\leb$-a.e.~$t$. For each $t$, select $\sigma_t(n)$ to be a subsequence such that $\Vert \tilde{f}'(u_{\sigma_t(n)}'(t,\cdot)) \Vert_{H^1_0([0,2])}$ is bounded (\emph{a priori}, this subsequence depends on the point $t$, and it satisfies \eqref{eq:weak convergence time slice}).

\smallskip 

\noindent 2. \textit{Uniform convergence on time-slices}. Applying Morrey's embedding to the estimate \eqref{eq:weak convergence time slice}, the sequence $\{\tilde{f}'(u_{\sigma_t(n)}'(t,\cdot))\}_n $ is uniformly bounded in $L^\infty([0,2])$ and equicontinuous.%; the uniform boundedness relies also on $\tilde{f}'(u_n'(t,y))=0$ at $y=0$ from Corollary \ref{cor:vanish endpoints}. 
Thus, by the Ascoli--Arzel\`a Theorem, each such sequence admits a further uniformly convergent subsequence $\{\tilde{f}'(u_{\sigma_t(n)}'(t,\cdot))\}_n$; again, this subsequence depends on $t$. Since $\tilde{f}'(u_\tau') \to \tilde{f}'(u')$ $\leb$-a.e., it holds 
\begin{equation}\label{eq:unif conv of time slices}
    \lim_{n\to\infty} \Vert \tilde{f}'(u_{\sigma_t(n)}'(t,\cdot)) - \tilde{f}'(u'(t,\cdot)) \Vert_{L^\infty([0,2])} = 0. 
\end{equation}

\smallskip 

\noindent 3. \textit{Convergence as $\tau \to 0$}. By returning to \eqref{eq:the return to get bad mixed prod convergence}, we observe that for every $[y_0,y_1]\subset (0,2)$ 
\begin{equation*}
    \begin{aligned}
        \int_{y_0}^{y_1}\!\!-\!\tilde{f}'(u_\tau') \, (\partial_y u)_s \!+\! \int_{y_0}^{y_1} \!\!-\! \tilde{f}'(u_\tau')(u'\!-\!u'_\tau) \d y  \leq \big[-\!\tilde{f}'(u_\tau')(u\!-\!u_\tau) \big]_{y_0}^{y_1} \!+\! \int_{y_0}^{y_1} (u\!-\!u_\tau) \partial_y (\tilde{f}'(u_\tau')) \d y , 
    \end{aligned}
\end{equation*}
which, by then re-inserting into \eqref{eq:first bound with convexity final a.e. conv} and using the non-negativity of $\omega(u',u'_\tau)$, yields 
\begin{equation*}
     \begin{aligned}  \int_{y_0}^{y_1}\! \big( \tilde{f}(u')  \!-\! \tilde{f}(u_\tau')  \big) \d y + \big[-\!\tilde{f}'(u_\tau')(u\!-\!u_\tau) \big]_{y_0}^{y_1} \!+\! \int_{y_0}^{y_1} (u\!-\!u_\tau) \partial_y (\tilde{f}'(u_\tau')) \d y \geq \int_{y_0}^{y_1}\!\!-\!\tilde{f}'(u_\tau') \, (\partial_y u)_s. 
     \end{aligned}
    \end{equation*}
    We estimate the second term on the left-hand side as per \eqref{eq:one sided important ii} and get 
    \begin{equation*}
     \begin{aligned}  
     \int_{y_0}^{y_1}\!\!-\!\tilde{f}'(u_\tau') \, (\partial_y u)_s \leq &\int_{y_0}^{y_1}\!  \big( \tilde{f}(u') \!-\! \tilde{f}(u_\tau')  \big) \d y + \Vert u \!-\! u_\tau \Vert_{L^2([y_0,y_1])} \Vert \tilde{f}'(u_\tau') \Vert_{H^1_0([0,2])} \\ 
    &+ \Vert \tilde{f}'(u_\tau') \Vert_{H^1_0([0,2])} \Big( \Vert u \Vert_{L^2([y_1,2])} +  \Vert u \Vert_{L^2([0,y_0])} \Big)   \\ 
        & + \Vert \tilde{f}'(u_\tau') \Vert_{H^1_0([0,2])} \Vert u_\tau \Vert_{L^2([0,2])} \Big( \frac{\sqrt{2-y_1}}{\sqrt{y_1}} + \frac{\sqrt{y_0}}{\sqrt{2-y_0}} \Big); 
     \end{aligned}
    \end{equation*}
    unlike in Step 3 of the proof of Lemma \ref{lem:strong a.e. conv of AC part of derivative}, we have not integrated in time. Then, we pass to the limit along the subsequence $\{\sigma_t(n)\}_n$ of Step 2, for which \eqref{eq:weak convergence time slice} and \eqref{eq:unif conv of time slices} hold, and get 
      \begin{equation*}
     \begin{aligned}  
  0 \leq \int_{y_0}^{y_1}\!\!-\!\tilde{f}'(u') \, (\partial_y u)_s = \lim_{n\to\infty}   \int_{y_0}^{y_1}\!\!-\!\tilde{f}'(u_{\sigma_t(n)}') \, (\partial_y u)_s 
  \leq   & \, C_t \Big( \Vert u \Vert_{L^2([y_1,2])} +  \Vert u \Vert_{L^2([0,y_0])} \Big)   \\ 
        & + C_t \Big( \frac{\sqrt{2-y_1}}{\sqrt{y_1}} + \frac{\sqrt{y_0}}{\sqrt{2-y_0}} \Big). 
     \end{aligned}
    \end{equation*}
    By letting $y_0 \to 0^+$ and $y_1 \to 2^-$, the right-hand side vanishes, which proves the claim. 
\end{proof}

Theorem \ref{thm:lagrangian} now follows directly from Proposition \ref{prop:existence L2 grad flow cts time} and Lemma \ref{lem:cts time optimality condition bad product}.

\section{Existence of Segregated Solutions of the Cross-Diffusion System}\label{sec:segregated cross diff}

We define $u_t(y) := u(t,y)$ where $u$ is the solution of \eqref{eq:final weak form continuous time} provided by Proposition \ref{prop:existence L2 grad flow cts time}, and fix a precise representative defined everywhere and equal to $u_t$ $\leb$-a.e., such that $\tilde{f}'(u'_t) \in C_0([0,2])$ by Morrey's embedding and the Fundamental Theorem of Calculus in $BV_\loc$ \eqref{eq:BV FTC} is satisfied. Henceforth, we do not distinguish $u_t$ from this representative. 

\begin{proof}[Proof of Theorem \ref{thm:existence}]
The proof is divided into four steps. 

\smallskip 

\noindent 1. \textit{Well-defined inverse function}. By Proposition \ref{prop:existence L2 grad flow cts time}, $u_t$ is strictly increasing on $(0,2)$: for all $y_1<y_2$, it holds 
\begin{equation}\label{eq:to get the inverse}
    u_t(y_2) - u_t(y_1) = \int_{y_1}^{y_2} \, \partial_y u_t (y) \geq \int_{y_1}^{y_2} u_t'(y) \d y \geq c_t (y_2-y_1) > 0, 
\end{equation}
using \eqref{eq:essinf in cts time}. Therefore, its inverse function is well-defined and strictly increasing on the image set $u_t((0,2))$. We denote $F_t(x) := u_t^{-1}(x)$, with $F_t : u_t((0,2)) \to (0,2)$.
\smallskip 

\noindent 2. \textit{Defining $\varrho_t,\mu_t$}. Define the set $G_t := \{y \in [0,2]: \tilde{f}'(u_t'(y)) \neq 0 \} $, and note that, since 
$\tilde{f}'(s)<0$ whenever $s$ is finite, and $u'$ is a.e. finite, then $G_t$ is of full measure in $[0,2]$.

Since $\tilde{f}'(u_t')$ is continuous for $\leb$-a.e.~$t$, the set $G_t$ is open in $[0,2]$. It follows that $G_t$ can be written as a countable union of open intervals, and hence as a countable union of compact intervals. By Lemma \ref{lem:cts time optimality condition bad product}, we deduce that the support of $(\partial_y u)_s$ is contained in the complement of $G_t$, so
\begin{equation}\label{eq:u is AC in the good set}
    \partial_y u_t = u_t' \!\cdot\! \leb \quad \text{in } G_t \implies u_t \in AC_\loc(G_t) \qquad (\leb\text{-a.e.~}t), 
\end{equation}
where $AC_\loc$ on the set $G_t$, which is open but may not be connected, means $AC$ on every compact interval contained within $G_t$. Hence $u_t$ is differentiable $\leb$-a.e.~in $G_t$. Since $u_t$ is continuous and strictly increasing on $G_t$, it is an open map, whence the image $u_t(G_t)$ is open in $\mathbb{R}$. Furthermore, for all $x_1,x_2 \in u_t (G_t)$, the condition \eqref{eq:to get the inverse} and the monotonicity of $F_t$ implies 
\begin{equation*}
   0 \leq  \frac{F_t(x_1)-F_t(x_2)}{x_1 - x_2} \leq \frac{1}{c_t}, 
\end{equation*}
whence $F_t$ is Lipschitz and thus  differentiable $\leb$-a.e. Differentiating $u_t(F_t(x)) = x$ at $\leb$-a.e.~point $x$ in the open set $u_t(G_t)$, we get 
\begin{equation}\label{eq:rigorously getting deriv of Ft}
    \partial_x F_t(x) =\frac{1}{u_t'(F_t(x))} \qquad \leb\text{-a.e.~}x\in u_t(G_t); 
\end{equation}
this is a standard result for absolutely continuous functions. Consequently, \textit{cf.}~\eqref{eq:cdf}, we set 
\begin{equation}\label{eq:definition of St}
    S_t(x) := \partial_x F_t(x) \mathds{1}_{u_t(G_t)} = \frac{1}{u_t'(F_t(x))}\mathds{1}_{u_t(G_t)} \qquad \forall x \in \mathbb{R}, 
\end{equation}
and we remark $S_t \in L^\infty(\mathbb{R})$ and $S_t(x) \!\cdot\! \leb = (u_t)_\sharp \leb \mres G_t$. In accordance with \eqref{eq:how rho and mu depend on b ish}, define 
\begin{equation}\label{eq:rhot mut derivs ac parts}
  \varrho_t(x) := \frac{1}{u'_t(F_t(x))}\mathds{1}_{u_t(A \cap G_t)} , \qquad   \mu_t(x) := \frac{1}{u'_t(F_t(x))}\mathds{1}_{u_t( A^c \cap G_t)}; 
\end{equation}
hence $\varrho_t(x) \!\cdot\! \leb = (u_t)_\sharp \leb \mres (A\cap G_t)$, while $\mu_t(x) \!\cdot\! \leb = (u_t)_\sharp \leb \mres (A^c \cap G_t)$, and $\varrho_t + \mu_t = S_t$.

\smallskip 

\noindent 3. \textit{$\partial_y \tilde{f}'(u'_t) \circ F_t = -\partial_x f'(S_t)$}. It follows from the relation \eqref{eq:f tilde deriv expressions} that, 
 for all $x \in u_t(G_t)$, 
\begin{equation*}
    \tilde{f}'(u_t')|_{y=F_t(x)} = f(S_t(x)) - S_t(x) f'(S_t(x)). 
\end{equation*}
By differentiating both sides with respect to $x$ in the set $u_t(G_t)$ (since absolutely continuous functions are differentiable $\leb$-a.e.), it follows that, for $\leb$-a.e.~$x \in u_t(G_t)$, 
\begin{equation*}
    \partial_y(\tilde{f}'(u_t'))|_{y=F_t(x)} \underbrace{\partial_x F_t}_{=S_t}(x) = \partial_x \big( f(S_t(x)) - S_t(x) f'(S_t(x)) \big) = -S_t(x) \partial_x (f'(S_t(x))). 
\end{equation*}
When $x \in u_t(G_t)$, \eqref{eq:definition of St} and  \eqref{eq:essinf in cts time} imply $S_t(x)>0$, whence we divide by $S_t(x)$ and get 
\begin{equation}\label{eq:converting from lagrangian to eulerian}
    \partial_y(\tilde{f}'(u_t'))|_{y=F_t(x)}  = -\partial_x (f'(S_t(x))) \qquad \leb\text{-a.e.~}x \in u_t(G_t).  
\end{equation}
It follows from the above, the change of coordinates $y=F_t(x)$, and \eqref{eq:rigorously getting deriv of Ft}, that 
\begin{equation*}
   \begin{aligned}
       \int_{u_t(G_t)} S_t(x) |\partial_x (f'(S_t(x)))|^2 \d x = \int_{u_t(G_t)} \frac{|\partial_y(\tilde{f}'(u_t'))|^2(F_t(x))}{u'_t(F_t(x))} \d x = \int_{G_t} |\partial_y (\tilde{f}'(u_t'))|^2 \d y, 
   \end{aligned}
\end{equation*}
whence, given that $S_t \equiv 0$ outside of $u_t(G_t)$ by definition \eqref{eq:definition of St}, integrating in time gives 
\begin{equation*}
    \Vert \sqrt{S_t}\partial_x(f'(S_t)) \Vert_{L^2([0,T]\times\mathbb{R})} = \Vert \tilde{f}'(u_t') \Vert_{L^2(0,T;H^1([0,2]))}, 
\end{equation*}
where we used from Step 2 that $G_t$ is of full measure to have equality.

\smallskip 

\noindent 4. \textit{Equations for $\varrho_t,\mu_t$}. Let $\varphi \in C^1_c(\mathbb{R})$. By definition \eqref{eq:rhot mut derivs ac parts} and the pushforward relation, 
\begin{equation*}
 \begin{aligned}
      \frac{\der}{\der t} \int_\mathbb{R} \varphi(x) \varrho_t(x) \d x \!=\! \frac{\der}{\der t} \int_{A\cap G_t} \varphi(u_t(y)) \d y \!=\! \frac{\der}{\der t} \int_{A} \varphi(u_t(y)) \d y \!=\! \int_{A \cap G_t} \partial_x \varphi(u_t(y)) \partial_t u_t(y) \d y. 
 \end{aligned}
\end{equation*}
where we use that $G_t$ is of full measure from Step 2 to repeatedly change the domain of integration from $A \cap G_t$ to $A$ and back. Recall that $u$ satisfies \eqref{eq:final weak form continuous time} $\leb$-a.e., whence 
\begin{equation*}
    \begin{aligned}
        \frac{\der}{\der t} \int_\mathbb{R} \varphi(x) \varrho_t(x) \d x  &= \int_{A\cap G_t} \partial_x \varphi(u_t(y)) \Big( \partial_y(\tilde{f}'(u_t'(y))) - b(y,u_t(y)) \Big) \d y  \\ 
        &= \!\int_{A \cap G_t}\!\!\!\!\! \partial_x \varphi(u_t(y))\partial_y(\tilde{f}'(u_t'(y))) \d y \!-\!\! \int_{\mathbb{R}}\! \partial_x \varphi(x) \underbrace{b(F_t(x),x) \mathds{1}_{\{F_t(x) \in A \}}}_{=\partial_x V(x)} \varrho_t(x) \d x \\ 
        &= - \int_{u_t(A\cap G_t)} \partial_x \varphi(x) S_t(x) \partial_x(f'(S_t(x))) \d x - \int_\mathbb{R} \partial_x \varphi(x) \varrho_t(x) \partial_x V(x) \d x, 
    \end{aligned}
\end{equation*}
where we used \eqref{eq:converting from lagrangian to eulerian} and the Jacobian factor $(u_t)_\sharp \d y \mres (A \cap G_t) = F_t'(x) \d x \mres u_t(A \cap G_t) = S_t(x) \d x \mres u_t(A \cap G_t)$ to obtain the final line. Rewriting in terms of $\varrho_t$, we find 
\begin{equation*}
    \begin{aligned}
        - \int_\mathbb{R} \partial_x \varphi(x) S_t(x) \mathds{1}_{\{F_t(x) \in A\cap G_t\}}\partial_x(f'(S_t(x))) \d x = - \int_\mathbb{R} \partial_x \varphi(x) \varrho_t(x) \partial_x(f'(S_t(x))) \d x, 
    \end{aligned}
\end{equation*}
and the equation \eqref{eq:weak form def} for $\varrho_t$ follows. The computation for $\mu_t$ is analogous, replacing $A$ with $A^c$ and $V$ with $W$ in the relevant manipulations. The convergence to the initial data in $W_2(\mathbb{R})$ follows from the strong $L^2$ convergence $u_t \to u_0$. 
\end{proof}

\smallskip

\noindent\textbf{Acknowledgements.} FS acknowledges support from the European Union via ERC AdG
101054420 EYAWKAJKOS project. SMS ackowledges support from CRM De Giorgi (SNS), where this work began, and thanks A.~Arroyo Rabasa and I.~Y.~Violo for useful discussions. 
Both authors also warmly acknowledge the support of the Lagrange Mathematics and Computation Research Center which hosted important discussions on this project.

\end{document}